\documentclass{amsart}
\usepackage{graphicx,amssymb,amsfonts,amsmath,amsthm,newlfont,epsfig}
\newtheorem{thm}{Theorem}[section]
\newtheorem{cor}[thm]{Corollary}
\newtheorem{lem}[thm]{Lemma}
\newtheorem{prop}[thm]{Proposition}

\theoremstyle{definition}
\newtheorem{defn}[thm]{Definition}

\theoremstyle{remark}
\newtheorem{rem}[thm]{Remark}
\numberwithin{equation}{section}

\newenvironment{proofofp}{{ \it Proof of proposition }}{\hfill $\square$}

\newcommand{\Z}{\mathbb Z}
\newcommand{\C}{\mathbb C}
\newcommand{\Pro}{\mathbb P}

\newcommand{\ai}{\mathfrak{a}}
\newcommand{\n}{\mathfrak{n}}

\newcommand{\bi}{\mathfrak{b}}

\newcommand{\R}{\mathbb R}
\newcommand{\N}{\mathbb N}
\newcommand{\Q}{\mathbb Q}

\newcommand{\Hom}{\hbox{Hom}}

\newcommand{\Rp}{\mathcal{R}}

\newcommand{\mot}{\mathsf{mot}}

\newcommand{\Hyp}{\mathbb H}
\newcommand{\Euc}{\mathbb E}

\newcommand{\X}{\mathbb X}

\newcommand{\SO}{\mathsf{SO}}
\newcommand{\GL}{\mathsf{GL}}
\newcommand{\PSL}{\mathsf{PSL}}

\newcommand{\U}{\mathsf{U}}

\newcommand{\quat}{\mathfrak{H}}

\newcommand{\A}{\mathbb {A}}

\newcommand{\Or}{\mathcal{O}}

\newcommand{\Fo}{\mathcal{F}}

\newcommand{\id}{\mathrm{id}}

\newcommand{\comp}{\mathrm{comp}}

\newcommand{\Mod}{\mathfrak{M}}

\newcommand{\F}{\mathbb F}

\newcommand{\To}{\longrightarrow}

\newcommand{\Gal}{\mathrm{Gal}}
\newcommand{\Sym}{\mathfrak{S}}

\newcommand{\Image}{\mathrm{Im\,}}
\newcommand{\Real}{\mathrm{Re\,}}

\newcommand{\vol}{\mathrm{vol}}

\newcommand{\MT}{\mathrm{MT}}

\newcommand{\gr}{\mathrm{gr}}
\newcommand{\Am}{\mathfrak{A}}

\newcommand{\Ext}{\mathrm{Ext}}

\newcommand{\Lo}{\mathcal{L}}

\newcommand{\K}{\mathbb{K}}
\newcommand{\Up}{\mathbb{U}}

\begin{document}

\title[Dedekind zeta motives for totally real  fields.]{  Dedekind zeta motives for totally real  number fields}%
\author{Francis C.S. Brown}%
\address{CNRS, IMJ}
\email{brown@math.jussieu.fr}

\maketitle

\vspace{-0.2in}

\begin{abstract}
Let $k$ be a totally real number field. For every odd $n\geq 3$, we construct an element  in the category $\MT(k)$
of mixed Tate motives over $k$ out of the quotient of a product of hyperbolic spaces by  an arithmetic group.   By  a volume calculation, we prove that its period is  a  rational multiple of $\pi^{n[k:\Q]}\zeta^*_k(1-n)$, where $\zeta^*_k(1-n)$ denotes
the special value of the Dedekind zeta function of $k$.
We deduce that
 the group $\Ext^1_{\MT(k)} (\Q(0),\Q(n))$ is generated by the cohomology of a quadric relative to hyperplanes, and that 
  $\zeta^*_k(1-n)$ is a determinant of volumes of geodesic hyperbolic simplices defined over $k$.
\end{abstract}

\section{Introduction}
\subsection{Background}
This paper was motivated by a desire to reconcile two important and complementary  results due to Zagier and Goncharov. Firstly, in 1986, Zagier obtained a formula for the value $\zeta_K(2)$, where $K$ is any number field, as a linear combination of products
of values of the dilogarithm function at algebraic points \cite{Z}.  The idea is to compute the volume of an arithmetic hyperbolic manifold in two different ways. Suppose that $K$ is quadratic imaginary,
and let $\Gamma$ be a torsion-free subgroup of finite index of the Bianchi group $\PSL_2(\Or_K)$, where $\Or_K$ is the ring of integers of $K$. Then $\Gamma$ acts  on
 hyperbolic 3-space $\Hyp^3$, and a classical theorem due to Humbert states that
$$\vol(\Hyp^3/\Gamma)={|d|^{3/2}\,r \over 4\pi^2}\zeta_K(2)\ ,  $$
where $d$ is the discriminant of $K$, and $r$ is the index of $\Gamma$ in $\PSL_2(\Or_K)$.
On the other hand,   such  a  hyperbolic 3-manifold can  be triangulated using  geodesic simplices,  whose volume, by a calculation originally due to  Milnor,
 can be expressed in terms of the Bloch-Wigner dilogarithm function
$$D(z)= \Image (\mathrm{Li}_2(z) + \log|z| \log(1-z))\ , \hbox{ where } \mathrm{Li}_2(z) = \sum_{k\geq 1} {z^k \over k^2}\ .$$
 The volume of $\Hyp^3/\Gamma$ is therefore an integer linear combination of $D(z)$ with algebraic  arguments.  Equating the two
volume calculations  yields Zagier's formula for $\zeta_K(2)$, which  in turn motivated his conjectures relating polylogarithms to special values of zeta and $L$-functions \cite{ZG}.  He also obtained a formula for $\zeta_K(2)$ for any number field
 by considering groups acting on products of hyperbolic 3-space in \cite{Z}. However, this geometric argument is insufficient to prove the final form of Zagier's conjecture for two reasons: one is that  the dilogarithms are evaluated
over some finite extension of $K$, rather than $K$ itself, and secondly there is no way to prove that the sums of products of dilogarithms are 
a determinant, as the conjecture demands.

 Zagier's results were subsequently reinterpreted in terms of algebraic K-theory by many   authors. See,  for example,  \cite{Bl,ZG,NY, NZ,Th} and the references therein.
Essentially, the set of   simplices in a  triangulation  of $\Hyp^3/\Gamma$
gives  a well-defined element in the Bloch group $B(\overline{\Q})$
via their gluing equations, and in turn defines an  element in the $K$-theory group $K_3(\overline{\Q})\otimes \Q$. The dilogarithm can be interpreted as (an explicit multiple of) the  Borel regulator map \cite{Bu}.

Secondly, in \cite{Go1},  Goncharov generalized these constructions for any discrete torsion-free group $\Gamma$ of finite covolume acting on hyperbolic
 $n$-space $\Hyp^n$, for $n=2m-1$ odd. For any such
 manifold $M=\Hyp^n/\Gamma$, he constructed an element
$$\xi_M \in K_{2m-1}(\overline{\Q})\otimes \Q\ ,$$
such that the volume of $M$ is given by the regulator map on $\xi_M$. In fact, he gave two such constructions, but it is the motivic version which is of interest
here.
The main idea is that a finite geodesic simplex in hyperbolic space (with algebraic vertices) defines a mixed Tate motive in the sense of \cite{DG}.   By triangulating the manifold $M$, and combining
the motives of each simplex, one obtains a total motive for $M$, which is non-trivial if $n$ is odd.  He then proves that the  motive of $M$ is an extension of $\Q(0)$ by $\Q(m)$, and identifies the  group
of such extensions  in the category $MT(\overline{\Q})$ of mixed Tate motives over $\overline{\Q}$ with $K_{2m-1}(\overline{\Q})\otimes \Q$ (see below).
The proof of this fact is analytic: it uses the fact that  (solid) angles around faces in a triangulation  sum to multiples of $2\pi$, and uses the full faithfullness
of the Hodge realization for mixed Tate motives  over $\overline{\Q}$ (\cite{DG}, \S2.13).

In this paper, we give a common generalization of both Zagier and Goncharov's results. 
 First  we  list the arithmetic groups which act on products of
hyperbolic spaces $\Hyp^{n_1}\times \ldots  \times \Hyp^{n_N}$, and give their covolumes, up to a rational multiple, in terms of zeta functions 
using well-known Tamagawa number arguments \cite{On1}. 
Next, we show that any such complete `product-hyperbolic' manifold $M$ of finite volume admits a triangulation using products
of geodesic simplices, and use this to construct a well-defined mixed Tate motive $\mot(M)$.
In the case when $N=1$, and $M$ is non-compact,  our construction differs from Goncharov's. 
We then prove that
\begin{equation}\label{introMinext1}
\mot(M) \in \Ext_{\MT(\overline{\Q})}^1(\Q(0),\Q(n_1)) \otimes_{\Q} \ldots \otimes_{\Q} \Ext_{\MT(\overline{\Q})}^1(\Q(0), \Q(n_N))\end{equation}
using a  geometric argument. Thus the construction is completely motivic.
Now let $k$ be a totally real number field, and $n$ an odd integer $\geq 3$. The Dedekind zeta motive is obtained in the case
 where $\Gamma$ is a torsion-free subgroup of the restriction of scalars  $R_{k/\Q}\SO(2n-1,1)$ acting on $(\Hyp^{2n-1})^{[k:\Q]}$, and its  period is directly related to $\zeta_k(n)$.
 Using  the action of a symmetric group on the framing of $\mot(M)$, we show that it is an alternating element of $(\ref{introMinext1})$, and applying Borel's calculation of the rank of algebraic K-groups, we deduce that $\mot(M)$
is a determinant of  $ \Ext_{\MT(k)}^1(\Q(0),\Q(n))\cong K_{2n-1}(k)\otimes \Q$.  As a corollary, we obtain  generators for $K_{2n-1}(k)\otimes \Q$ in terms of
hyperbolic geometry
and deduce  that   $\zeta_k(n)$ is a determinant of volumes of hyperbolic simplices.

\subsection{Explicit categories of mixed Tate motives and zeta values}\label{sectintroexpl}
 Let $k$ be a number field, and  let $\MT(k)$ be the abelian tensor category  of mixed Tate motives over $k$  \cite{DG}. Its  simple objects are the Tate motives
$\Q(n)$, for $n\in \Z$, and its structure  is determined by its relation to algebraic $K$-theory:
\begin{equation}\label{introext1}
\Ext^1_{\MT(k)}(\Q(0),\Q(n)) \cong \left\{
                             \begin{array}{ll}
                               0 & \hbox{ if  } \ n\leq 0\ ,  \\
                               K_{2n-1}(k)\otimes \Q & \hbox{ if }\  n\geq 1 \ ,
                             \end{array}
                           \right.
\end{equation}
and the fact that all higher extension groups vanish.  One has $K_1(k)=k^\times$, and a theorem due to Borel \cite{Bo1},
states that  for $n>1$,
\begin{equation} \label{introKtheoryrank}
\dim_{\Q} (K_{2n-1}(k)\otimes \Q) =\left\{
                                     \begin{array}{ll}
                                     n_+=  r_1+r_2, & \hbox{ if  } n \hbox{ is odd } ; \\
                                      n_-= r_2, & \hbox{ if } n \hbox{
is even}\ ,
                                     \end{array}
                                   \right.
\end{equation}
where $r_1$ is the number of real places of $k$, and $r_2$ is the
number of complex places of $k$.
The Hodge realisation functor on $\MT(k)$ gives rise to a regulator map
\begin{equation} \label{introhodge}
r_H:\Ext^1_{\MT(k)}(\Q(0),\Q(n)) \To \R^{n_\pm}  \qquad (n>1) \ . 
\end{equation}
  Its covolume should be related to the  values of the Dedekind zeta function $\zeta_k(s)
$ at $s=n$ or $s=1-n$ via the Beilinson and Bloch-Kato conjectures \cite{Be}, \cite{BK}.
\vspace{0.1in}

It is an important open problem  to construct  $\MT(k)$ explicitly 
out of simple geometric building blocks.  Concretely, one  writes down a category 
$C\subseteq \MT(k)$ constructed from certain diagrams of   varieties with Tate cohomology, and one wants to show  that $C=\MT(k)$. To do this, firstly one must 
 show that   the groups  $\Ext^1_{\MT(k)}(\Q(0),\Q(n))$ are  spanned by elements of $C$ (generation), and secondly, one must show
that $C$ is closed under extensions (freeness).  Via such a construction  one hopes to compute the periods of $\MT(k)$, and these include the numbers $\zeta_k(n)$ for $n \geq 2$, in particular.  Natural candidates for $C$ come from smooth hypersurfaces with Tate cohomology (or their blow-ups along linear subspaces), as follows.
\vspace{0.1in}

\emph{Hyperplanes.}
  The  approach of \cite{BGSV} uses 
 arrangements of hyperplanes in projective space. Consider   a set of $2n+2$ hyperplanes
$L_0,\ldots, L_n,M_0,\ldots, M_n\subset \Pro^n$ in general position, and write $L=\cup_{i=0}^n L_i$, $M=\cup_{i=0}^n M_i$.  Let $k$ be a number field.  If $L,M$ are defined over $k$,  the   pair $(L,M)$ defines a mixed Tate motive:
\begin{equation} \label{introhyperplanemotive}
H^n\big(\Pro^{n}\backslash M, L\backslash (L\cap M)\big)\in \MT(k)\ ,
\end{equation}
with weights in $[0,2n]$. Let $C^h(k)\subset \MT(k) $ denote the full subcategory spanned by the elements $ (\ref{introhyperplanemotive})$ and closed under tensor products, subquotients and Tate twists. The  first  part of the problem described above is to show that the natural map
\begin{equation}\label{introphih}
 \phi^h:\Ext^1_{C^h(k)} (\Q(0),\Q(n))  \To \Ext^1_{\MT(k)}(\Q(0),\Q(n))\ 
\end{equation}
is surjective. Stefan M\"uller-Stach informed us  that this follows from the work of Gerdes \cite{Ge} (conjecture 4.7, remark 4.11) along
with the  proof of the rank conjecture for algebraic $K$-theory \cite{Y}.
A  variant  of the  above construction  would be to  consider blow-ups  of degenerate configurations of hyperplanes. The  motivic fundamental group  of the punctured line minus roots of unity  \cite{DG} is a special case of this latter construction  in the particular case when $k$ is abelian.
\vspace{0.1in}

\emph{Quadrics and hyperplanes \cite{Go1}.}
  Let $Q$ denote a smooth quadric in $\Pro^{2n-1}$, and let $L_1,\ldots, L_{2n}$
denote $2n$ hyperplanes in general position with respect to $Q$.  Let $L=\cup_{i=1}^{2n} L_i$. 
 If  $Q$ and $L$ are defined over $\overline{\Q}$, this defines a \emph{quadric motive}:
\begin{equation} \label{introquadricmotive}
H^{2n-1}\big(\Pro^{2n-1}\backslash Q, L\backslash (Q\cap L)\big)\in \MT(\overline{\Q})\ ,
\end{equation}
with weights in $[0,4n-2]$.
As above, let $C^q(\overline{\Q})\subset \MT(\overline{\Q}) $ denote the full subcategory spanned by the elements $ (\ref{introquadricmotive})$, and let 
$C^q(k) \subset \MT(k) $ be the subcategory which is  invariant under $\Gal(\overline{k}/k)$.
The generation part of the problem is now to show that 
\begin{equation}\label{introphiq}
 \phi^q: \Ext^1_{C^q(k)} (\Q(0),\Q(n))  \To \Ext^1_{\MT(k)}(\Q(0),\Q(n))
\end{equation}
is surjective.  In this paper, we shall prove this  when $k$ is  totally real. 
\vspace{0.1in}

A  much  more precise statement should follow from    Beilinson and Deligne's reformulation \cite{BD} of Zagier's conjecture  \cite{ZG}.
For every $z\in k\backslash \{0,1\}$, and $n\geq 1$, one can construct  a polylogarithm motive
$P_n(z)\in \MT(k)$ as a subquotient  of a (degenerate) hyperplane motive  $(\ref{introhyperplanemotive})$. If $K_z$ denotes the Kummer motive (extension of 
$\Q(0)$ by $\Q(1)$) corresponding to   $\log z$, then $P_n(z)$ is part of an exact sequence
$$0 \To \big(\mathrm{Sym}^{n-1} K_z\big) (1) \To P_n(z) \To \Q(0)\To 0$$
and satisfies $P_n(z)/W_{-4} P_n(z) \cong K_{1-z}$.  The objects $P_n(z)$ span a  strict subcategory $C^p(k)\subset \MT(k)$   and the classical polylogarithm $\mathrm{Li}_n(z)=\sum_{k\geq 1} {z^k\over k^n}$ is a period of $P_n(z)$.
A version of  Zagier's conjecture  would state that 
\begin{equation}\label{introphip}
\phi^p: \Ext^1_{C^p(k)}(\Q(0),\Q(n)) \To \Ext^1_{\MT(k)}(\Q(0),\Q(n))\ .
\end{equation}
 is surjective, i.e.,  
$(\ref{introext1})$ should be generated by polylogarithm motives.   This is known only for $n\leq 3$ \cite{Go3} but is supported by extensive numerical evidence 
\cite{ZG}.

\subsubsection{Zeta values}
The surjectivity of $(\ref{introphih})$, $(\ref{introphiq})$, or  $(\ref{introphip})$  would  imply a statement about values of zeta functions, as  follows.
Let $\zeta_k(s)$ denote the Dedekind zeta function of $k$, and let $\zeta^*_k(1-n)$ denote the leading coefficient in the Taylor expansion of $\zeta_k(s)$ at $s=1-n$, for $n\geq 2$.
The image of the  Borel regulator
\begin{equation}\label{introBorelreg}
r_{B}:K_{2n-1}(k)\otimes \Q \To \R^{n_{\pm}}\qquad (n>1)\  ,\end{equation}
is a $\Q$-lattice $\Lambda_n(k)$ whose covolume is well-defined up to multiplication by an element in $\Q^\times$.
Using the isomorphism  of rational $K$-theory with the stable cohomology of the linear group,  Borel proved by an analytic argument  in \cite{Bo2} that
\begin{equation} \label{introBorelzeta}
\zeta^*_k(1-n) = \alpha\, \mathrm{covol} (\Lambda_n(k)) \quad \hbox{ for some } \alpha\in \Q^{\times} \ .
\end{equation}
If one combines this  result with the isomorphism $(\ref{introext1})$,   then  the  surjectivity of one of the maps
$\phi^{\bullet}$ above yields a conjectural  formula for $\zeta_k^*(1-n)$ modulo rationals (assuming the compatibility between the various regulators involved).
 In the case of the map $\phi^h$, this would express
$\zeta_k^*(1-n)$ as a determinant of Aomoto polylogarithms \cite{BGSV}, and in the case of the map
$\phi^p$, this would express
$\zeta_k^*(1-n)$ as a determinant of classical polylogarithms \cite{ZG}.   

\subsection{Main Results} Let $k$ be a totally real number field.  If $n$ is even, the group $(\ref{introext1})$ vanishes. For every odd $n\geq 3$, we construct
a canonical Dedekind zeta motive
\begin{equation} \label{introxidef}
\mot_n(k) \in \MT(k)\ ,\end{equation}
using arithmetic subgroups of  $\SO(2n-1,1)$ (see below).
The motive $\mot_n(k)$ is a  subquotient of a  sum   of  products of quadric motives $(\ref{introquadricmotive})$.

\begin{thm} \label{Intromaintheorem}
The element $\mot_n(k)$  satisfies
$$\mot_n(k) \in \textstyle{\det}\,\,\Ext^1_{\MT(k)} (\Q(0),\Q(n))\ ,$$
and its image under the regulator  $(\ref{introhodge})$  is a non-zero rational multiple  of $\zeta^*_k(1-n)$.
\end{thm}

Note that the regulator $(\ref{introhodge})$ is normalised so that the powers of $\pi$ drop out in the previous theorem.
Theorem $\ref{Intromaintheorem}$ is a motivic analogue of Borel's theorem $(\ref{introBorelzeta})$
in the totally real case. However, the proof of theorem $\ref{Intromaintheorem}$ is entirely different, as it uses $r_H$ rather than $r_B$, and works completely inside the category $\MT(k)$. This opens up
the possibility of studying  other realisations of $\mot_n(k)$.

\begin{thm}  \label{Intromaintheorem2} If $k$ is totally real, the map $\phi^q$ of $(\ref{introphiq})$ is surjective.
\end{thm}
Thus every element in $\Ext^1_{\MT(k)}(\Q(0),\Q(n))$, for $k$ a totally
real number field, is a subquotient of a direct sum  of quadric motives.
Theorems $\ref{Intromaintheorem}$ and $\ref{Intromaintheorem2}$ also hold for certain quadratic extensions of $k$. A  similar construction when $n=2$ holds  in fact for any number field.

The element $\mot_n(k)$ is defined in the following way. Let $\Hyp^m$ denote hyperbolic space of dimension $m$, and
let $\Or_k$ denote the ring of integers of $k$. Then  any torsion-free subgroup  $\Gamma$ of finite index of $\SO(2n-1,1)(\Or_k)$
acts properly discontinuously on $r=[k:\Q]$ copies of $\Hyp^{2n-1}$.  The quotient is a \emph{product-hyperbolic} manifold
$$M= \Hyp^{2n-1} \times \ldots \times \Hyp^{2n-1} / \Gamma \ .$$
Then, using a trick due to Zagier, one can show that $M$ can be triangulated using products of hyperbolic geodesic simplices defined over $k$.
By an idea due to Goncharov, a hyperbolic geodesic simplex in the Klein model for $\Hyp^{2n-1}$ is a Euclidean simplex $L$ inside a smooth quadric $Q$, and so
defines a quadric motive $(\ref{introquadricmotive})$. The total motive $\mot(M)$ is defined to be a certain subquotient of the direct sum of  tensor products of quadric  motives 
corresponding to all the simplices in the triangulation of $M$, and is well-defined. Technically, this subquotient is extracted using  
framed equivalence classes (\S5.1).
 The gluing relations between simplices imply that $\mot_n(k)\in \bigotimes_{[k:\Q]} \Ext^1_{\MT(k)}(\Q(0),\Q(n))$,  generalizing
a result due to Goncharov for ordinary hyperbolic manifolds. A further symmetry argument is required to show that $\mot_n(k)$ is alternating.
Theorem $\ref{Intromaintheorem}$ is then proved using the fact that
$$\vol(M) = \alpha\, \pi^{nr}\zeta_k^*(1-n)\ ,$$
for some  $\alpha\in\Q^\times$, which follows from  a Tamagawa number argument.
This in turn implies theorem $\ref{Intromaintheorem2}$, using the fact that the rank of $\Ext^1_{\MT(k)}(\Q(0),\Q(n))$ is exactly $r$,
which follows from $(\ref{introKtheoryrank})$. 
There are  two corollaries.

\begin{cor} Let $M\in \Ext^1_{\MT(k)}(\Q(0),\Q(n)).$ The periods of $M$ under $\pi^nr_H$ are $\Q$-linear combinations of volumes of hyperbolic $(2n-1)$-simplices
defined over $k$.
\end{cor}

\begin{cor} The special value  $\zeta^*_k(1-n)$ is  a rational multiple times a determinant of ($\pi^{-n}$-times) volumes of hyperbolic $(2n-1)$-simplices defined over $k$.
\end{cor}

This result is in the spirit of, but weaker than, Zagier's conjecture.
The proof  that the zeta value is a determinant uses the rank calculation $(\ref{introKtheoryrank})$ in an essential way, and does not seem to follow directly from the geometry
of the triangulation.
  A similar argument to the above also works for $n=2$, and holds for   all number fields $L$. It is not hard to show that quadric motives in $\Pro^3$ are dilogarithm motives  $P_2(z)$,
thereby giving another
proof of Zagier's conjecture on zeta values for $n=2$.

\subsection{Plan}
In $\S2$ we introduce notations and define product-hyperbolic manifolds.
In $\S3$, we  list the non-exceptional arithmetic product-hyperbolic manifolds, and state their
covolumes up to a rational multiple. In $\S4$, we explain how to decompose a product-hyperbolic manifold $M$ into a finite union  of products of hyperbolic geodesic simplices $ \Delta^{(i)}_1 \times \ldots \times \Delta^{(i)}_N$, $1\leq i \leq R$, 
where each simplex $\Delta^{(i)}_j$ has at most one vertex at infinity.
In $\S5$, we  recall  properties of framed mixed Tate motives and study the framed mixed Tate motive defined by such a geodesic simplex and its periods, with examples in 
$\S\ref{Sectexamples3space}$.
The total (framed) motive $\mot(M)$ is defined  by 
\begin{equation}\label{intromotdef}
\mot(M) = \sum_{i=1}^R \mot(\Delta^{(i)}_1)\otimes \ldots \otimes \mot(\Delta^{(i)}_N)\ .\end{equation}
and is well-defined and non-zero if $M$ is modelled on products of hyperbolic spaces with  odd-dimensional components only.  In $\S\ref{sectcoproductvanishes}$,  we prove  $(\ref{introMinext1})$ from a geometric argument which uses only the fact that $M$ has no boundary.
Theorem  $\ref{thmmotiveMAIN}$ covers the case when  $M$ is defined arithmetically. In $\S6$,  these results are combined 
 to prove the main theorems.  Analogous
results in the exceptional case $n=2$ are given in $\S\ref{sectZagiersconj}$, and some open questions  are discussed in $\S\ref{sectOpen}$.
Finally, in $\S7$ we compute
an explicit
example of the motive of an Artin $L$-value at $3$ over the field $\Q(\sqrt{5})$, which is  based on a remarkable computation due to Bugaenko. 
\vspace{0.1in}

This work was written during my thesis at the University of Bordeaux, in 2004-2005. Many thanks to 
C. Soul\'e for his comments on  a very early version of this paper, and also to H. Gangl and M. Belolipetsky for helpful discussions.
At the time of publication, I am partially supported by ERC grant number 257638.

\section{Product-hyperbolic manifolds.} 

\subsection{Euclidean and hyperbolic spaces}
We use the following notations.
For $n \geq 1$,  Euclidean space $\Euc^n$ is $\R^n$ equipped with the scalar product
 $$(x,y)= x_1 y_1 +\ldots +x_ny_n\ .$$
Let $n\geq 2$, and let $\R^{n,1}$ denote  $\R^{n+1}$ equipped with the inner product
$$(x,y)= -x_0y_0+ x_1y_1+\ldots +x_ny_n\ .$$
Hyperbolic space $\Hyp^n$ is defined to be the half-hyperboloid:
$$\Hyp^n= \{ x\in \R^{n,1}: (x,x)=-1, \,\, x_0>0\}\ .$$
Let $\SO(n,1)(\R)$ denote the group of matrices preserving this scalar
product. It has two components. Let $\SO^+(n,1)(\R)$ denote
the connected
component of the identity, which is the group of
orientation-preserving symmetries of $\Hyp^n$.
 The invariant metric on $\Hyp^n$ is given by $ds^2=-dx_0^2+dx_1^2 +\ldots +dx_n^2$.

 We will consider  complete manifolds $M$
which are locally modelled on products of the above, {\it i.e.}, Riemannian products of the form
\begin{equation}\label{Xtype}
\X^{\n}=\prod_{1\leq i \leq N} \Hyp^{n_i}\times \prod_{N+1 \leq i\leq N'} \Euc^{n_i} \ ,
\end{equation} 
where $\n=(n_1,\ldots, n_{N'})$.  A   complete,
orientable manifold modelled on $\X^\n$ of finite volume is of the form
 $M = \X^\n/\Gamma,$
where $\Gamma$ is a discrete torsion-free subgroup of the group of motions of $\X^\n$.
 In this case, we say that $M$ is a \emph{flat-hyperbolic} manifold.
If it is modelled only on products of hyperbolic spaces ($N'=N$), it will be called \emph{product-hyperbolic}.
It follows from Margulis' theorem on the arithmeticity of lattices of rank $>1$ \cite{Marg} that any
irreducible discrete group of finite covolume $\Gamma$ acting
on $\X^\n$ is  arithmetic, or else $\X^\n= \Hyp^{n}$ and
$\Hyp^n/\Gamma$ is an ordinary hyperbolic manifold.

\subsection{Models of hyperbolic spaces}
 We also need to consider the following models for hyperbolic space (\cite{AVS}, I,\S2) and the absolute, which we denote by $\partial \Hyp^n$.

\subsubsection{The Klein (projective) model}
The map $(x_0,\ldots, x_n) \mapsto (y_1,\ldots, y_n)$, where $y_i=x_i/x_0$, gives an isomorphism of the hyperboloid model with
the Klein model:
\begin{equation} \label{Kmodel} \K^n=\{(y_1,\ldots, y_n)\in \R^n:
r^2=\sum_{i=1}^n y_i^2 <1\}
\end{equation}
 equipped with the
metric $ds^2 = (1-r^2)^{-2}[ (1-r^2)\sum_{i=1}^n dy_i^2 +
(\sum_{i=1}^n y_i dy_i)^2]$.
 The action of $\SO^+(n,1)$ on $\K^n$ is by projective
transformations, and extends continuously to the boundary
$\partial \K^n$, which is the unit sphere in $\R^n$. Geodesics in this model are Euclidean
lines,  and in particular, Euclidean hyperplanes are totally geodesic, but note that 
 hyperbolic and Euclidean angles do not agree. 
\subsubsection{The Poincar\'e upper-half space
model}  \label{sectPUPM} Let
\begin{equation} \label{Umodel} \Up^n=\{(z_1,\ldots,z_{n-1}, t) :\, (z_1,\ldots,
z_{n-1})\in \R^{n-1}, \,t>0 \}\ ,
\end{equation}
equipped with the metric $t^{-2} (\sum_{i=1}^{n-1} dz_i^2 +dt^2).$ The
absolute $\partial \Hyp^n$ is identified with the Euclidean plane
$\partial \Up^n = \R^{n-1}\times \{0\} \subset \R^{n-1}\times \R_{\geq 0}$ at height $t=0$, compactified by
adding the single point at infinity $\infty$. Geodesics in this
model are vertical line segments (which go to $\infty$) or
segments of  circles which meet the absolute at right angles. In
this model, hyperbolic angles coincide with Euclidean angles.
\section{Volumes of product-hyperbolic manifolds and  $L$-functions}

\subsection{Arithmetic  groups acting on product-hyperbolic space}\label{sect4types}
There are four basic types of arithmetic groups which act on products of hyperbolic spaces.
\\

\textbf{{\it Type (I)\ :\,}} Let $k/\Q$ be a totally real number field of degree $r$, and let $\Or$ denote an order in $k$. 
  Let $1\leq t\leq r$. Consider a non-degenerate quadratic form
  \begin{equation} \label{qdef} q(x_0,\ldots,x_n) = \sum_{i,j=0}^n a_{ij} x_i x_j\quad \hbox{where } \quad a_{ij} \in k\ ,\,\, a_{ij}=a_{ji}\ ,\end{equation}
  and suppose that $^{\sigma}\! q=\sum_{i,j=0}^n {}^\sigma \!a_{ij}\, x_i x_j$
    has signature $(n,1)$ for $t$ infinite places $\sigma\in\{\sigma_1,\ldots, \sigma_t\}$ of $k$, and is positive definite
  for the remaining $r-t$ infinite places of $k$. Let
  $\SO^+(q,\Or)$ be the group of  linear transformations with coefficients in
  $\Or$ preserving $q$, which map each connected component of $\{x\in \R^{n+1}:
  \,{ }^{\sigma_i}\!q(x)<0\}$ to itself for $1\leq i\leq t$.
  Let $\Gamma$ be a  torsion-free subgroup of  $\SO^+(q,\Or)$ of finite index.
 It acts properly discontinuously on $\prod_{i=1}^t \Hyp^n$ via the map
\begin{eqnarray}
\Gamma &\hookrightarrow& \prod_{i=1}^t \SO^+(n,1)(\R) \nonumber \\
A & \mapsto & ( \,{ }^{\sigma_1}\!A,\ldots,  \,{ }^{\sigma_t}\!A)\ . \nonumber
\end{eqnarray}

\textbf{{\it Type (II)\ :\,}}
Suppose that $n=2m-1\geq 3$ is odd. Let $k/\Q$ be a totally real number field of degree $r$, and let $1\leq t\leq r$.
   Consider a quaternion  algebra $D$ over $k$ such that $D_{\sigma}=D\otimes_{k,\sigma}\R$
   is  isomorphic to $M_{2\times 2}(\R)$ for all
   embeddings  $\sigma$ of $k$. Let
\begin{equation}\label{Qdef}
Q(x,y)= \sum_{i,j=1}^{m} \overline{x}_i a_{ij} y_j \quad \hbox{ where } \quad  a_{ij}
\in D, \,\, a_{ij} = -\overline{a}_{ji}\ , \end{equation} be a non-degenerate
skew-Hermitian form on $D^{m}$, where $x\mapsto \overline{x}$
denotes the conjugation map on $D$, which is an anti-homomorphism. For each embedding $\sigma$ of $k$,  the signature
of ${}^\sigma \!Q$, where ${}^\sigma \!Q(x,y)= \sum_{i,j=1}^{m} \overline{x}_i \,{}^\sigma  \!a_{ij} y_j$, is defined as follows.
Let
$D_\sigma\cong M_{2\times 2}(\R)= \R\oplus \mathsf{i}\R \oplus \mathsf{j}\R \oplus \mathsf{k}\R$, where
$\mathsf{i}^2=\mathsf{j}^2=1$ and $\mathsf{i}\mathsf{j}=-\mathsf{j}\mathsf{i}=\mathsf{k}$. The endomorphism of $D_\sigma$ given by right
multiplication by $\mathsf{i}$ is  of order two and has eigenvalues $\pm1$.
Let
$$D_{\sigma,\pm}= \{ x\in D: x\mathsf{i}=\pm x\}\ .$$
It follows that $D^m_\sigma= D_{\sigma,+}^m \oplus D_{\sigma,-}^m$, and because one has $D_{\sigma,-}^m = D_{\sigma,+}^m
\mathsf{j}$, the dimension of $D^m_{\sigma, +}$ is $2m$. Writing $x=x_++x_-$, where $x_+\mathsf{i}=x_+$ and $x_-\mathsf{i}=-x_-$, one verifies
that
${}^\sigma \!Q(x_+,y_+)={}^\sigma \!Q(x_+\mathsf{i},y_+)= -\mathsf{i}\, {}^\sigma \!Q(x_+,y_+)$,  and  ${}^\sigma \!Q(x_+,y_+)={}^\sigma \!Q(x_+,y_+\mathsf{i})=
{}^\sigma \!Q(x_+,y_+)\mathsf{i}$.
 It follows that
$${}^\sigma \!Q(x_+,y_+) = f_\sigma(x_+,y_+) (\mathsf{j}-\mathsf{k})\ , \quad \hbox{where }f_\sigma(x_+,y_+)\in k\ .$$
Since  ${}^\sigma \!Q$ is  skew-hermitian, $f_\sigma$ is a non-degenerate symmetric bilinear form on $D^m_{\sigma,+}$. The signature of ${}^\sigma \!Q$ is defined to be the signature
of $f_\sigma$. Let $\phi_\sigma$ denote the map which to
$Q$ associates the bilinear form $f_\sigma$.
The form $f_\sigma$ uniquely determines ${}^\sigma \!Q$, for instance, via
 the formula $-\mathsf{j}\, {}^\sigma \!Q(x_-,y_+)= {}^\sigma \!Q(x_-\mathsf{j}, y_+)= f_\sigma(x_-\mathsf{j}, y_+)(\mathsf{j}-\mathsf{k})$.

Suppose, therefore,  that $^\sigma\! Q$ has
signature $(2m-1,1)$ for $t$ embeddings $\sigma \in \{\sigma_1,\ldots,\sigma_t\}$ of $k$, and is positive
definite for  the remaining $r-t$  embeddings.
 Let $\Or$ be an
order in $D$, and let $\U^+(Q,\Or)$ denote
the group of $\Or$-valued points of the unitary group
 which preserves the connected components of $\{x\in
D^m_{\sigma_i,+} : f_{\sigma_i}(x,x)<0\}$ for $1\leq i\leq t$.
If  $\Gamma$ is a torsion-free subgroup of $\U^+(Q,\Or)$, then
it  acts properly
discontinuously on $\prod_{i=1}^t \Hyp^n$ via the map
$(\phi_{\sigma_1},\ldots, \phi_{\sigma_t}):\Gamma \hookrightarrow \prod_{i=1}^t \SO^+(n,1)(\R).$
\\

\textbf{{\it Type (III)\ :\,}}
Let $L/\Q$ denote a number field of degree $n$ with $r_1$ real places and $r_2\geq 1$ complex
  places, and let $0\leq t \leq r_1$. Let $B$ denote a quaternion algebra over $L$ which is
  unramified at $t$ real places, and ramified at  the
  other $r_1-t$ real places of $L$.
Then $\prod_{v|\infty} (B\otimes_L L_v)^* = (\quat^*)^{r_1-t}
\times \GL_2(\R)^t \times \GL_2(\C)^{r_2}$, where $\quat$ denotes Hamilton's quaternions.
  Let $\Or$ denote an order in
  $B$ and let $\Gamma$ be a torsion-free subgroup of finite index
  of the group of elements in $B$ of reduced norm $1$.  Then $\Gamma$ defines
a discrete subgroup of $\PSL_2(\R)^{t}\times \PSL_2(\C)^{r_2}$ and acts properly  discontinuously on $(\Hyp^2)^t
\times (\Hyp^3)^{r_2}$ \cite{Ch}.
\\

\textbf{{\it Type (IV)\ :\,}} There are further exceptional cases for $\Hyp^7$  which are related 
to Cayley's octonions. These will not be considered here.

\begin{rem}\label{remtype1istype2} When $n$ is odd, every arithmetic group of type $(I)$ can also be expressed as an arithmetic group of type $(II)$. So in $(I)$ we can assume $n$ is even.

 Our main results  only use the trivial  case $q=-x_0^2+x_1^2+\ldots + x_n^2$. Then
we obtain the
standard orthogonal group $\SO(n,1)$, and $\Gamma$ is  any
torsion-free subgroup of finite index of $R_{k/\Q} \SO^+(n,1)(\Z) = \SO^+(n,1)(\Or_k)$.
In the case $n=2$, the main example is given by $B=M_{2\times 2}(L)$, where $L$ is a  number field as in $(III)$.  Then $\Gamma$ is  a torsion-free subgroup of finite
index of the Bianchi group $\PSL_2(\Or_L)$.
\end{rem}

\subsection{Covolumes of arithmetic product-hyperbolic manifolds.}
For $k$  a number field, let $\zeta_k(s)$ denote the Dedekind
zeta function of $k$, and let $d_k$ denote the absolute value of
the discriminant of $k$. If $\chi$ is the non-trivial character of
a quadratic extension $L/k$, let $L(\chi,s)=
\zeta_L(s)/\zeta_k(s)$ denote its Artin
$L$-function, and let $d_{L/k}=d_L/d_k^2$. For any $\alpha,\alpha'  \in \R$, we write $\alpha \sim_{\Q^\times} \alpha'$ if $\alpha =\beta  \alpha'$ for some $\beta \in \Q^{\times}$. 

\begin{thm} \label{thmvolumeslist} The covolumes of   arithmetic groups of types $(I)- (III)$
acting on products of hyperbolic spaces  are as  follows.

\begin{enumerate}
  \item Let $n=2m,$ let $k$ be a totally real  field of degree $r$, and let $\Gamma$ be of type (I) acting on $\prod_{i=1}^t \Hyp^n$, where
   $1\leq t \leq r$. Let $M=(\prod_{i=1}^t \Hyp^n)/\Gamma$. Then
$$
\vol(M)  \sim_{\Q^\times}
     |d_k|^{n(n+1)/4}\,
   \pi^{-m^2r+m(t-r)} \, \zeta_k(2)\zeta_k(4) \ldots \zeta_k(2m)\ ,
$$

  \item Let $n=2m-1\geq 3$ be odd and let $k$ denote a totally real field of degree $r$. Let $\Gamma$ be  of type (II) acting on $\prod_{i=1}^t \Hyp^n$ for
  some  $1\leq t\leq r$, defined in terms of a skew-Hermitian form $Q$ over a quaternion algebra $D$.
 Let $d\in k^\times/k^{\times 2}$
  denote the reduced norm of the discriminant of $Q$, and let $\chi$ denote the
  non-trivial character of the quadratic extension
  $L=k(\sqrt{d})$ of $k$. Let $L(\chi,m) = \zeta_{L}(m) \zeta_k(m)^{-1}$.
If $M= (\prod_{i=1}^t \Hyp^n)/\Gamma$, then
$$\vol(M)
\sim_{\Q^\times}
 \left\{
  \begin{array}{ll}
   |d_{L/k}|^{1/2}\,|d_k|^{n(n+1)/4}  \pi^{-m^2r+mt} \zeta_k(2) \zeta_k(4) \ldots
\zeta_k(2m-2)\, L(\chi,m) , & \hbox{ if  } [L:K]=2\ , \\
   |d_k|^{n(n+1)/4}  \pi^{-m^2r+mt} \zeta_k(2) \zeta_k(4) \ldots
\zeta_k(2m-2)\, \zeta_k(m)  , & \hbox{ if  } L=K\ ,
  \end{array}
\right.
 $$
In the special case where $\Gamma$ is of type $(I)$, the same formula holds, where $d$ is the discriminant of the quadratic form $q$ of $(\ref{qdef})$.
\\

  \item Let $L$ be a number field with $r_1$ real places and $r_2\geq 1$
  complex places. Let $\Gamma$ denote an arithmetic group of type (III)
   acting on $\X=\prod_{i=1}^t \Hyp^2 \times \prod_{j=1}^{r_2}
  \Hyp^3$ where $1\leq t\leq r_1$. If $M=\X/\Gamma$, then
$$\vol(M)  \sim_{\Q^\times} |d_L|^{3/2}
\pi^{t-2r_1-2r_2}\zeta_L(2)\ . $$
\end{enumerate}
\end{thm}

In order to reduce the length of the paper, we give some general indications for deriving  the formulae above, without giving a full proof.  More generally, let $G$ be a connected semisimple 
algebraic group defined over a number field $k$, such that for at least one infinite place $v$ of $k$, $G(k_v)$ is not compact.  Let $K$ denote a maximal compact subgroup of $G$, and let $X_v=K(k_v)\backslash G(k_v)$ be the corresponding symmetric space for each infinite place $v$ of $k$. Set
 $X=\prod_{v|\infty}ÊX_v$, and let $\Gamma$ denote a torsion-free subgroup of finite index of
the group of $\Or_k$-valued points of $G$, which acts properly
discontinuously on $X$. In   \cite{On1} it was shown  that 
 $$\vol(X/\Gamma) \sim_{\Q^\times} |d_k|^{\,\dim G/2} \Delta_{G} \, \, \Big(\prod_{v|\infty} \vol(K_v)^{-1}\Big)\,  \prod_{v<\infty} {q^{\,\dim G}_v\over \#G(\F_{v}) } \ ,$$
for the standard invariant volume forms on $X$  and $K_v$ (which are canonical up to a sign). Here, $\F_v$ denotes the residue field of $k$ at a finite place $v$, and  $q_v=\# \F_v$.  In the cases we are interested in, $K_v$ is an orthogonal group, whose volume is easily computed and yields a power of $\pi$. The  infinite product on the right converges (\cite{On3}, appendix II), and the point counts of $G$ over finite fields are well-known \cite{On1}. More precisely, in case (I)  of \S\ref{sect4types}, $G$ is of type $B_{n}$ and we have
$$   \#G(\F_{v}) = q^{m^2} \prod_{k=1}^m (q^{2k}-1)\ .$$
In case (II),   $G$ is of type $D_m$ and we have
$$\#G(\F_v) =  \left\{
  \begin{array}{ll}
   q_v^{m(m-1)}(q_v^m-1) \prod_{k=1}^{m-1} (q_v^{2k}-1)    & \hbox{ if  } d\in \F_v^2\ , \\
   q_v^{m(m-1)}(q_v^m+1) \prod_{k=1}^{m-1} (q_v^{2k}-1)   & \hbox{ if  } d\notin \F_v^2\ .
  \end{array}
\right.
$$
Finally, in case (III), $G$ is of type $A_1$ and $\#G(\F_v) = q(q^2-1)$. This yields the product of  $L$-values in the volume formula. Finally, the invariant
$\Delta_{G}$ is in $\Q^{\times}$ when $G$ is an inner form of $SO(n+1)$ (cases (I), (III) and case (II) when $d\in k^2$) or $d_{L/k}^{1/2} \Q^{\times}$  if it is a non-trivial
outer form (case (III) when  $d\notin k^2$).

\subsection{Summary of volume computations}
One can show by a non-compact version of the  Gauss-Bonnet formula \cite{Ha}, or by using Poincar\'e's  formula for the volumes of even-dimensional simplices and a geodesic triangulation (e.g. \cite{R} \S11.3),  that if 
 $M$ is a product-hyperbolic manifold of finite volume
 modelled on products of even-dimensional spaces, i.e.,  $M= \prod_{i\in I} \Hyp^{2n_i} /\Gamma$, then we have:
 \begin{equation} \label{evenvolume}
 \vol(M) \sim_{\Q^\times} \pi^{\sum_{i\in I} n_i}.
 \end{equation}

 This  implies the following theorem 
due to Siegel and Klingen.
\begin{cor} If $k$ is a totally real field,  $\zeta_k(1-2m) \in \Q^{\times}$ for all
integers $m\geq 1.$
\end{cor}

\begin{proof} Let $n=2m$, and let $\Gamma\leq \SO^+(n,1)(\Or_k)$ be a
torsion-free subgroup of finite index, where $k$ is totally real of degree $r$.  Applying $(\ref{evenvolume})$ to
$M=(\Hyp^n)^r/\Gamma$  gives $\vol(M)\in \pi^{mr}\Q^{\times}$. Theorem $\ref{thmvolumeslist}$  with $t=r$ implies that
$$\vol(M) \sim_{\Q\times} |d_k|^{n(n+1)/4} \pi^{-rm^2} \zeta_k(2)\ldots \zeta_k(2m) \ .$$
The functional equation for
$\zeta_k(s)$ gives
$\zeta_k(1-2m)\sim_{\Q^{\times}} d^{\,(2m-1)/2}_k\pi^{-2mr}\zeta_k(2m)$, 
 and hence
$\zeta_k(-1)\,\zeta_k(-3)\ldots  \zeta_k(1-2m) \in
\Q^{\times}$, for all $m\geq 1$.  
\end{proof}

\begin{cor} \label{corvolumeslist} If an arithmetic product-hyperbolic manifold is modelled on even-dimensional spaces, its volume is a rational multiple of a power of $\pi$.
In the case where $n=2m-1$ is odd, or $n=2$, the  formulae of theorem $\ref{thmvolumeslist}$ simplify to
$$ \pi^{mt} \,\sqrt{|d_k|}\,\Big(\displaystyle{\zeta_k(m) \over \pi^{m[k:\Q]} }\Big)\ ,\quad  \pi^{mt} \, \sqrt{|d_{L/k}d_k|} \, \Big(\displaystyle{ L(\chi,m) \over \pi^{m[k:\Q]} }\Big)
\ ,\quad  \pi^{t+2r_2} \, \sqrt{|d_L|} \Big(\displaystyle{ \zeta_L(2) \over \pi^{2[L:\Q]}}\Big)\ .$$
By the corresponding functional equations, these are non-zero rational multiples of, respectively 
\begin{equation}\label{Finalvolumeformula}
 \pi^{mt} \,\zeta^*_k(1-m)\ ,\quad  \pi^{mt}  L^*(\chi,1-m)
\ ,\quad  \pi^{t+2r_2} \,  \zeta^*_L(-1)\  .\end{equation}

\end{cor}

\section{Triangulation of Product-hyperbolic manifolds.}

In order to define the motive of a product-hyperbolic manifold, we must   decompose it
 into compact and cuspidal parts, and  triangulating each part with products of simplices
 using an inclusion-exclusion argument due to Zagier.

\subsection{Decomposition of product-hyperbolic manifolds into cusp sectors}
The cusps of a product-hyperbolic manifold are best described using
the
  upper-half space model $\Up^n\cong \{ (z,t) \in \R^{n-1} \times \R_{>0}\}$ (\S2).
For all $r>0$,  let $B_n(r)\subset \Hyp^n$ denote the closed
horoball near the point at infinity:
$$B_n(r) = \{(z,t)\in \R^{n-1}\times \R_{>0}: \, t\geq r\}.$$
Its boundary, the horosphere,  is isometric to a Euclidean plane
$\Euc^{n-1}$ at height $r$. 
 Now let
$\X^\n= \prod_{i\in I} \Hyp^{n_i}$. For any subset $\emptyset\neq S\subset I$, and
any set of parameters $\underline{r}=\{r_i>0\}_{i \in S}$ indexed by
$S$, we define the corresponding product-horoball to be
$$B_S(\underline{r})=\prod_{i\in S} B_{n_i}(r_i) \times \prod_{i\in
I\backslash S} \Hyp^{n_i}\subset \X^\n.
 $$
Its horosphere is  $\prod_{i\in S} \partial B_{n_i}(r_i) \times
\prod_{i\in I\backslash S} \Hyp^{n_i}$ which is  diffeomorphic 
to the  space $\prod_{i\in S} \Euc^{n_i-1} \times
\prod_{i\in I\backslash S} \Hyp^{n_i}.$
The cusps of a
product-hyperbolic manifold are  diffeomorphic to a product
$F\times \R_{>0}^{|S|}$ 
 where $F$ is a compact flat-hyperbolic manifold.

\begin{thm} \label{thmdecompintoends} Any   complete product-hyperbolic manifold $M$ of finite volume has a decomposition
into finitely many disjoint pieces
$$M=\bigcup_{\ell=0}^p M_{(\ell)}$$
where $M_{(0)}$ is compact, and  each  $M_{(\ell)}$ for $\ell\geq 1$ is isometric to a cusp, {\it i.e.},
$M_{(\ell)}\cong F_\ell \times \R_{>0}^{k_\ell} ,$ where $k_{\ell}\geq1 $ and
 $F_\ell$ is a compact flat-hyperbolic manifold with boundary.
\end{thm}

The proof of the theorem is different in the case when $M$ is
arithmetic or non-arithmetic. By Margulis' theorem, an irreducible
non-arithmetic lattice is necessarily of rank one and acts on a
single hyperbolic space. In this case, $M$ is an ordinary hyperbolic
manifold of finite volume and the decomposition is well-known (\cite{BaGS}, III, \S10 and appendix).
In the case of arithmetically-defined groups, the decomposition
follows from \cite{L-M}, Proposition 4.6, or the main result of \cite{Sa}.

\subsection{Geodesic simplices and rational points}\label{sectrationalpoints}

\subsubsection{Rational points on product-hyperbolic manifolds}
Let $k$ be a subfield of $\R$, and let  $\Hyp_k^n$ (resp.   $\Euc_k^n$)  denote the set of points with $k$-rational coordinates
in  the hyperboloid model for hyperbolic space  (resp. Euclidean  space).  The former coincides with the set of $k$-rational points in the upper-half space model, but 
is strictly contained in the  set of $k$-points in the Klein model.  Define $\partial \Hyp^n_k$ to be the set of $k$-rational points on the absolute $k^{n-1}\times \{0\}\subset \partial \Up^n=\R^{n-1} \times \{0\}$ (\S\ref{sectPUPM}), and write  $\overline{\Hyp}^n_k = \Hyp^n_k \cup \partial \Hyp^n_k$.
 The sets  $\Hyp^n_k$ and $\partial \Hyp^n_k$
are  preserved by $\SO^+(n,1)(k)$ and are dense in $\Hyp^n$ and $\partial \Hyp^n$ respectively.
 Now
consider  products of hyperbolic spaces $\X^\n=\prod_{i=1}^N \Hyp^{n_i}$. Let $S=(k_1,\ldots, k_N)$
denote   a  tuple  of fields, with $k_i\subset \R$.
 We define the set of $S$-rational points on $\X^\n$ to be $\X^\n_S=
\prod_{i=1}^N \Hyp^{n_i}_{k_i}$. We  say that a product-hyperbolic
manifold $M=\X^\n/\Gamma$ is \emph{defined over
$S$} if 
 $\Gamma$ is conjugate to $\Gamma'$ which satisfies
$$\Gamma' \leq \prod_{1\leq i\leq N} \SO^+(n_i,1)(k_i)\ .$$

\begin{rem}
In dimension $n=3$, there is an exceptional isomorphism
between $\SO^+(3,1)(\R)$ and $\PSL_2(\C)$. For any  field $L\subset \C$, we define the set of $L$-points $\Hyp^3_L$
of $\Hyp^3$ to be the  orbit of  a fixed rational point of $\Hyp^3$ under
$\PSL_2(L)$.  From now on, we assume  for simplicity that  $\X$  has no such hyperbolic components in dimension $3$ and all fields $k_i$ are real. The translation to the  case  $\Hyp^3_L$ is straightforward.
\end{rem}

\begin{thm}\label{thmdefinedoverfield}
Let $M$ be a complete product-hyperbolic manifold of finite volume.
Then $M$ is defined over  $S=(k_1,\ldots, k_r)$ where $k_i$ are (embedded)  number
fields.
\end{thm}

\begin{proof}
For arithmetic
 groups
 this follows from the definition. For non-arithmetic groups, it
suffices  to consider discrete subgroups of $\SO^+(n,1)$, by
Margulis' arithmeticity theorem. By Weil's rigidity theorem, $\Gamma$ is conjugate to
a subgroup $\Gamma'$ of $\SO^+(n,1)(\overline{\Q})$. Since
$\Gamma'$ is finitely generated, its entries  lie in a field
$k\subset \R$ which is finite over $\Q$. \end{proof}

We
 define the set of $S$-rational points of such an  $M$ defined over $S$
to be
\begin{equation}
M_S = \X^\n_S/\Gamma'\ .\end{equation} The set of $S$-rational
points $M_S$ is
 dense in $M$.

\begin{lem}\label{lemequivcusps}
The fixed point $z\in \partial\Hyp^n$ of any parabolic motion
$\gamma\in \SO^+(n,1)(k)$ is defined over $k$, {\it i.e.}, $z\in
\partial\Hyp^n_k$.
\end{lem}

\begin{proof}
In the Klein model for hyperbolic space, the
 group $\SO^+(n,1)(k)$ acts 
by projective transformations on the absolute $\partial \Hyp^n=\{z_1,\ldots,z_n\, : \sum_i z_i^2=1\}$. 
It follows that a point $z \in \partial \Hyp^n$ which is stabilised
by $\gamma \in \SO^+(n,1)(k)$ satisfies an equation $\gamma z=z$,
which is algebraic in the coordinates of $z$, and has coefficients
in $k$. If $\gamma $ is parabolic, then this equation has a unique
solution on the absolute.  By uniqueness, $z$ coincides with its conjugates under
$\Gal(\overline{k}/k)$,  and hence $z \in
\partial \Hyp^n_k$.
\end{proof}

\subsubsection{The  action of the symmetric group} \label{sectsymandequiv}
Now let $\Sigma=\{\sigma_1,\ldots, \sigma_N\}$ denote a set of
distinct real embeddings of a fixed field $k$ as above. Set $k_i=\sigma_i k$, and let
$S=(k_1,\ldots, k_N)$.   For each pair of indices $1\leq i,j\leq
N$, there is a bijection
$\sigma_i \,\sigma_j^{-1} : \overline{\Hyp}^n_{k_j}
\overset{\sim}{\To} \overline{\Hyp}^n_{k_i}.$ Let
  $\Sym_N$ denote the symmetric group on $N$ letters
 $\{1,\ldots, N\}$.

\begin{defn} If the dimensions of all hyperbolic components $n_i$ are equal to $n$, the symmetric group $\Sym_N$ acts on $\overline{\X}^\n_S= \prod_{i=1}^N
\overline{\Hyp}^{n}_{k_i}$ as follows:
$$
\pi(x_1,\ldots, x_N)  =  (\sigma_1 \sigma_{\pi(1)}^{-1}
x_{\pi(1)},\ldots, \sigma_N \sigma_{\pi(N)}^{-1} x_{\pi(N)})\ ,
\quad \hbox{ where } \pi \in \Sym_N\ . $$
\end{defn}

 We define the  \emph{equivariant points} of $\overline{\X}^\n_{S}$ to be the   fixed points under this  action.

\begin{defn}  A  product-hyperbolic manifold $M= \prod_{i=1}^N \Hyp^{n_i}/\Gamma$, with all $n_i$ equal to $n$, is
\emph{equivariant} with respect to $\Sym_N$ if $\Gamma$
 lies in the image
of
\begin{eqnarray}
e: \SO(n,1)(k) &\To & \prod_{1\leq i\leq N} \SO(n,1)({\sigma_i(k)}) \nonumber \\
A &\mapsto & (\sigma_1(A),\ldots, \sigma_N(A)) \ .\nonumber
\end{eqnarray}
Note that the fixed point of a parabolic
motion on an equivariant product-hyperbolic manifold  is necessarily equivariant.\end{defn}

\subsubsection{Geodesic simplices  and the action of $\Sym_N$.}
Let $\X$ be $\Euc^n$ or $\Hyp^n$, for  $n\geq 2$. If
$x_0,\ldots, x_n$ are $n+1$ distinct points in $\X$,  let
$\Delta(x_0,\ldots, x_n)$ denote the geodesic simplex whose vertices
are $x_0,\ldots, x_n$. This is defined  to be the
 convex hull of the points
$\{x_0,\ldots, x_n\}$. 
If the
points $x_0,\ldots, x_n$ lie in a proper  geodesic subspace, then
$\Delta(x_0,\ldots, x_n)$ will be degenerate. The boundary faces of $\Delta(x_0,\ldots, x_n)$ are the convex hulls
of  nonempty strict subsets of the points $\{x_{0},\ldots, x_n\}$. 
When $\X=\Hyp^n$, we can allow some or all of the
vertices $x_i$ to lie on the boundary $\partial \Hyp^n$. One can
show that such a  simplex $\Delta(x_0,\ldots, x_n)$ always has
finite, and in fact bounded, volume.
 The simplex
$\Delta(x_0,\ldots, x_n)$ is said to be \emph{defined over a field
$k\subset \R$}, if $x_0,\ldots, x_n\in \overline{\X}_k$.
 Suppose  that we are given a  map
of fields $\sigma: k\hookrightarrow k'$. It induces an action on the set of
geodesic simplices  defined over $k$:
\begin{equation} \label{Galoisonsimplices}
\sigma \Delta(x_0,\ldots, x_n) = \Delta (\sigma x_0,\ldots, \sigma
x_n) \subset \overline{\X}_{k'}\ .
\end{equation}
  In a product of spaces of type $(\ref{Xtype})$, we
consider products of the form
\begin{equation} \label{productsimplex}
\Delta= \Delta_1\times\ldots \times \Delta_N\ ,\end{equation}
where $\Delta_i$ are geodesic simplices in each component.
We call this a \emph{geodesic product-simplex}. It is \emph{defined
over the fields $S=(k_1,\ldots, k_N)$} if $\Delta_i$ is
defined over $k_i$ for all $1\leq i\leq N$. In the  equivariant case
$S=(\sigma_1 k,\ldots, \sigma_N k)$, and all $n_i$ are equal,
the symmetric group $\Sym_N$ acts on  geodesic product simplices
defined over $S$ as follows:
\begin{equation}\label{Equivgeodaction}
\pi(\Delta_1\times \ldots \times \Delta_N) =
\sigma_1\sigma^{-1}_{\pi(1)}\Delta_{\pi(1)} \times \ldots
\times\sigma_1\sigma^{-1}_{\pi(N)}\Delta_{\pi(N)}\, \hbox{ for any
} \pi \in \Sym_N\ .\end{equation}

\subsubsection{Cones over Euclidean simplices}\label{sectcones} Let  $\Delta(x_1,\ldots, x_n)\subset \Euc^{n-1}$ be a Euclidean geodesic simplex.
Let us identify the horosphere in $\Up^n$ at height $R>0$ with $\Euc^{n-1}\times \{R\} \subset \Up^n$. If $\infty$ denotes  the point at infinity,
let us write $\Delta_{\infty} = \Delta (x_1,\ldots, x_n,\infty)$ for the geodesic hyperbolic simplex  $\Delta \times [R,\infty)$ which is the cone over $\Delta$.

\subsection{Generalities on virtual triangulations}
Let  $M$ denote a  flat-hyperbolic manifold.
In order   to decompose $M$ into  products of geodesic simplices, we must
consider geodesic polytopes in products of Euclidean and hyperbolic
spaces, which may have vertices at infinity. Let $P$ denote an
$n$-dimensional convex polytope   in Euclidean or
hyperbolic space $\overline{\Hyp}^n$.
 A \emph{faceting} $F(P)$ of $P$  (\cite{R}, \S11.1) is a finite collection
of closed, $(n-1)$-dimensional convex polytopes
 $F_i$,  called \emph{facets}, which are contained in the boundary $\partial P$ of $P$ 
 such that:
\begin{enumerate}
  \item Each face of $P$ is a union of facets. 
  \item Any two facets are either disjoint or meet along their common boundaries.
\end{enumerate}
The set of all codimension $1$ faces of $P$, for example, defines a
faceting of $P$. A \emph{product-polytope} in $\overline{\X}^\n$ is
 a product of convex geodesic polytopes, and one defines a faceting
 in  a similar manner.
For any product-polytope $P$, let $P^f$ denote the polytope $P$ with
all its points which have a component on an absolute removed.
 Now let $R$ denote a finite set of
geodesic product-polytopes in $\overline{\X}^\n$, and
suppose we are given a faceting for each product-polytope in $R$. A
set of gluing relations for $R$ is   a way to  identify all facets
of all polytopes $P\in R$ in pairs which are isometric. Let
$$T= \coprod_{P\in R} P^f/\sim$$
 denote the topological space obtained by identifying glued facets. 
 A \emph{product-tiling} of $M$ is then an
isometry
\begin{equation} f: T \To M\ . \end{equation} Since facets may be
strictly contained in the faces of each geodesic simplex, the tiling
is not always a triangulation. See also \cite{R}, $\S11$ for a much more detailed treatment of tilings  and triangulations in a general context.

We will also need to consider \emph{virtual} tilings. To define a
virtual tiling, consider  a   continuous surjective map
$f: T \To M,$
of finite degree which is not necessarily \'etale. We assume that
 the restriction of $f$ to the interior of
each product polytope $P$ in $T$ is an isometry. We define the local
multiplicity
 of $f$ on $P$  to be $+1$ if
 $f|_P$ is orientation-preserving, and $-1$ if
 $f|_P$ is orientation-reversing. The condition that $f$ be a virtual tiling is that
the total multiplicity of $f$ is almost everywhere equal to $1$.
In the case when $M$ is defined over the fields
$S=(k_1,\ldots, k_N)$,  we  will say that the (virtual) product
tiling is \emph{defined over $S$} if the geodesic product simplices
which occur in $T$ are defined over  $S$.

\subsection{Tiling of product-hyperbolic manifolds}
One can construct a  product-tiling of  any complete, orientable,
finite-volume flat-hyperbolic manifold $M$, using a variant of an argument  due
to Zagier  \cite{Z}.

\begin{lem}
Let $X_1,\ldots, X_n$  and $Y_1,\ldots, Y_n$ denote any sets,  and
let $\epsilon \in \{0,1\}^n$. We write $\epsilon=(\epsilon_1,\ldots, \epsilon_n)$, and let
 $$X_i(\epsilon)= \begin{cases}
    X_i \cap Y_i & \text{if } \quad \epsilon_i=0\ , \\
    X_i \backslash (X_i \cap Y_i) & \text{if } \quad \epsilon_i=1\ ,
  \end{cases}
$$ and define $Y_i (\epsilon)$ similarly.
Any union of products $(X_1\times \ldots \times X_n) \cup (Y_1
\times \ldots \times Y_n)$ can be written
$$(X_1\cap Y_1) \times \ldots\times (X_n\cap Y_n) \cup \!\!\!\!\bigcup_{0\neq \epsilon \in \{0,1\}^n} X_1(\epsilon)\times \ldots \times X_n(\epsilon) \cup
\!\!\!\!\bigcup_{0\neq \epsilon \in \{0,1\}^n} Y_1(\epsilon)\times
\ldots \times Y_n(\epsilon), $$ where all unions are disjoint.
\end{lem}
 We  apply the lemma to a
pair  of  geodesic product-simplices in flat-hyperbolic space. Let
$\Delta = \prod_{i \in I} \Delta_{i}$ and $\Delta' = \prod_{i \in I}
\Delta_i'$ where $\Delta_i, \Delta_i' $ are geodesic simplices in
$\Euc^{n_i}$ or $\overline{\Hyp}^{n_i}$ for each $i\in I$. By the
lemma, $\Delta\cup \Delta'$ can be decomposed as a disjoint union of
 products of $\Delta_{i} \cap \Delta_i'$,  $\Delta_i\backslash (\Delta_i \cap \Delta_i')$, or $\Delta'_i\backslash (\Delta_i \cap \Delta'_i)$,
 for $i\in I$. In
each case, the intersection of two geodesic simplices (or its
complement) is a union of geodesic polytopes. Every such polytope can be
decomposed as a disjoint union of geodesic simplices by
subdividing and triangulating. It follows that we can triangulate
any overlapping union $\Delta\cup \Delta'$ with  geodesic
product-simplices.

\begin{cor} \label{corexcis}
Any finite union of product-simplices can be obtained by gluing
finitely many geodesic product-simplices along pairs of common
facets.
\end{cor}

We can now prove the main result of this section, following \cite{Z}.

\begin{prop}\label{proptriang} Every product-hyperbolic manifold $M$  of finite volume admits a finite tiling by products of  geodesic
simplices. If $M$ is defined over  a tuple of fields $S$ as in $\S\ref{sectrationalpoints}$,  then
we can assume the product simplices have $S$-rational vertices.
\end{prop}
\begin{proof} By theorem $\ref{thmdecompintoends}$, there is a finite decomposition $M= \bigcup_\ell M_{(\ell)},$
 where $M_{(0)}$ is compact and for  $\ell\geq 1$, $M_{(\ell)}$ is a cone over $F_\ell$, a
compact flat-hyperbolic manifold  with boundary. We first
tile each  $M_{(\ell)}$ for $\ell\geq 1$ by constructing a tiling of $F_\ell$ and  taking  the cone at infinity over this tiling using  $\S\ref{sectcones}$. We then obtain a tiling of $M$ on taking
the union of all the geodesic product-simplices involved  in the tiling of each piece $M_{(\ell)}$ and
excising the overlaps using  corollary $\ref{corexcis}$.

Let $M= \X^\n/\Gamma$ and
 let $\pi:\X^\n\rightarrow M$ denote the covering map.
Let  $\Fo\subset \X^\n$ be a fundamental set for $M$, with a  decomposition $\Fo=\bigcup_\ell \Fo_\ell$,
where $\Fo_0$ is compact, and each
 $\Fo_\ell$ for $\ell\geq 1$ is diffeomorphic  to  $\R^{k_\ell}_{>0} \times D_\ell$,  where $D_\ell$ is a compact domain in flat-hyperbolic space $\prod_{i\in S} \Euc^{n_i-1} \times
 \prod_{j\in I\backslash S} \Hyp^{n_j}$ (and hence $k_\ell=|S|$). Cover  $D_\ell$ with geodesic product-simplices as follows.
Since $D_\ell$ is compact, choose compact polyhedra $K_i
\subset \Euc^{n_i-1}$, $K_j \subset \Hyp^{n_j}$, where $i\in S,
j\in I\backslash S$, such that $D_\ell\subset \prod_{i\in S} K_i
\times \prod_{j\in I\backslash S} K_j$. We denote the restriction
of the covering map $D_\ell \rightarrow F_\ell$ by $\pi$ also.
Each set $K_i$ can be triangulated by finitely many  oriented geodesic
simplices $\Delta^{(i)}_{a_i}$ which are sufficiently small such
that any product $\Delta_{\ai}=\prod_{i \in S}\Delta_{a_i}^{(i)}
\times \prod_{j\in I\backslash S}\Delta_{a_j}^{(j)}$, where
$\ai=(a_i)_{i\in I}$, is mapped isometrically
 onto $\pi(\Delta_{\ai}).$ Let $\ai \neq \bi$ be indices such that
 the pair of simplices $\pi(\Delta_{\ai})$, $\pi(\Delta_{\bi})$ have
non-empty overlap. It follows that there is a $\gamma \in \Gamma$
such that $\gamma \Delta_{\ai}\cap \Delta_{\bi}\neq \emptyset$.
Applying   the previous corollary to the union $\gamma
\Delta_{\ai}\cup \Delta_{\bi}$,   we can replace $\Delta_{\ai}
\cup \Delta_{\bi}$ with a union of product-simplices whose
interiors are disjoint after projection down to  $F_\ell$. We can
repeatedly excise the overlap between simplices using the previous lemma and its corollary to obtain the
required tiling of $D_\ell$ with product-simplices. To show that
this process terminates after finitely many steps, let $d(x)\in
\N$ denote the multiplicity of the tiling at each point $x\in
D_\ell$. Since $\Fo$ is a fundamental set,  there is an integer
$N$ such that $d(x)\leq N$ for all $x\in D_\ell$. Every time an
excision is applied, the local multiplicities  $d(x)$ for all $x$
in some open subset of $D_\ell$ decrease by $1$. Since $D_{\ell}$
is compact, this process terminates when $d(x)=1$  almost
everywhere. This gives the required tiling of $D_\ell$.
Each Euclidean component of $D_{\ell}$ can be identified with a suitable horoball neighbourhood   of infinity in $\Up^n$.
By the construction of \S \ref{sectcones}, replace  every product of flat-hyperbolic
geodesic simplices that occurs in the tiling of $D_\ell$:
$$\prod_{i\in S} \Delta^{(i)}\times \prod_{j\in I\backslash S} \Delta^{(j)} \in \prod_{i\in S} \Euc^{n_i-1} \times
 \prod_{j\in I\backslash S} \Hyp^{n_j}\ ,$$
with its cone at infinity (\S \ref{sectcones}):
$$\prod_{i\in S} \Delta^{(i)}_{\infty} \times \prod_{j\in I\backslash S} \Delta^{(j)} \in \prod_{i \in S} \overline{\Hyp}^{n_i} \times
 \prod_{j\in I\backslash S} \Hyp^{n_j}\ .$$
This gives  a covering of  each end $\Fo_{\ell}$, for $\ell\geq 1$,  with product-simplices which maps via $\pi$ to a tiling of $M_{(\ell)}$. 
Finally,  the compact part $\Fo_{0}$ of $\Fo$ can be covered by a large compact set  $K_0$ which can be triangulated with geodesic product-simplices as before.
 All together, we have  a covering of $\Fo$ with  finitely many geodesic product-simplices
which maps to a tiling in each part $M_{(\ell)}$. 
 By applying corollary \ref{corexcis} as above, we can excise the overlaps between these simplices all over again.  Since the  $K_0\cap \Fo_{i}$  are compact,   we obtain a  tiling of $M$   after finitely many steps.

For the last statement, suppose that  $M$ is defined over a tuple of fields $S$, as in $\S\ref{sectrationalpoints}$.
 Since the  set of $S$-rational points is dense in $M$, and since by lemma
4.9,  the coordinates at infinity of the cusps of $M$ are
$S$-rational, we can ensure in the previous argument that all product-simplices have vertices defined over $S$.
\end{proof}

Note  that the geodesic simplices which occur have at most one vertex at infinity, a fact which will be used later.
 It does not matter if degenerate  simplices occur.

\subsection{Equivariance} Now suppose that the
product-hyperbolic manifold $M$ is equivariant (\S \ref{sectsymandequiv}). By proposition $\ref{proptriang}$, there exists a product-tiling of $M$ which is defined
over $S$, where $S=(\sigma_1k,\ldots, \sigma_N k)$.  Let $T$
denote the set of hyperbolic product-simplices in this tiling. The
symmetry group $\Sym_N$ acts on the  product-simplices, and
preserves the subdivisions of a face into its facets $F_i$: if $F$ is tiled by facets $F_i$ for $1\leq i \leq m$, then
${}^\sigma\!F$ is tiled by facets ${}^\sigma \!F_i,$
for all $\sigma \in \Sym_N$. Note that some of the simplices $^\sigma \!F_i$
may be oriented negatively, so the latter tiling is in fact   a virtual
tiling of ${}^\sigma \!F$. Let $^\sigma\! T$ denote the set of
images of the elements of $T$ under $\sigma \in \Sym_N$. We can glue
the simplices in $^\sigma\!T $ back together according to the same
gluing pattern to form a virtual tiling of $^\sigma \!M$. Since
$M$ was assumed to be equivariant, this gives a new tiling of
$M$, which is also defined over $S$.

\begin{lem} \label{lemequivtile} Let $M$ denote an equivariant product-hyperbolic
manifold. If $T$ is a product-tiling of $M$  defined over $S$,
then ${}^\sigma T$ is  a virtual product-tiling of $M$
defined over $S$, for all $\sigma \in \Sym_N$.
\end{lem}

Note that, given a finite virtual tiling of
$M$, one can obtain a genuine tiling
by
subdividing and excising any overlaps using a variant of  corollary $\ref{corexcis}$.

\section{The motive of a product-hyperbolic manifold}

\subsection{Framed Mixed Tate motives and their periods.}
\subsubsection{Mixed Tate motives} Let $k$ be a number field, and let $\MT(k)$ denote the abelian tensor
category of mixed Tate motives over $k$ \cite{DG}. Its simple objects are the pure Tate motives
 $\Q(n)$, where $n\in \Z$, and its structure is determined by:
\begin{equation}\label{extgroups}
\Ext_{\MT(k)}^1(\Q(0), \Q(n)) \cong \left\{
                             \begin{array}{ll}
                               0 & \hbox{ if  } n\leq 0\ ,  \\
                               K_{2n-1}(k)\otimes \Q\ , & \hbox{ if }
n\geq 1\  , 
                             \end{array}
                           \right.
\end{equation}
and the fact that $\Ext_{\MT(k)}^2(\Q(0), \Q(n))=0$ for all $n\in \Z$. 
Recall that  $K_1(k)\cong k^\times$, and  Borel proved in \cite{Bo1}   that for $n>1$ ,
\begin{equation} \label{Ktheorydimensions}
\dim_{\Q} (K_{2n-1}(\Q)\otimes \Q) = \left\{
                                     \begin{array}{ll}
                                       r_1+r_2, & \hbox{ if  } n \hbox{ is odd } ; \\
                                       r_2, & \hbox{ if } n \hbox{
is even}\ ,
                                     \end{array}
                                   \right.
\end{equation}
where $r_1$ is the number of real places of $k$, and $r_2$ is the
number of complex places of $k$. Although we do not explicitly
require $(\ref{Ktheorydimensions})$ to prove theorem $\ref{Intromaintheorem}$, it
is used implicitly in the construction of $\MT(k)$.

Every mixed Tate motive $M\in \MT(k)$ has a canonical weight
filtration, which is increasing, finite, and indexed by even integers.  For each $n\in \Z$, its graded piece of weight $-2n$ is denoted  by $\gr^W_{-2n} M=W_{-2n}M / W_{-2n-1} M$,  and is isomorphic to a finite sum of copies of $\Q(n)$.  Following \cite{DG}, one sets
$$\omega_n(M)= \Hom(\Q(n), \gr^W_{-2n} M)\ .$$
 The de
Rham realisation of $M$ is defined to be the graded vector space
 $M_{DR}=\omega(M)\otimes_\Q k$, where $\omega(M)=\oplus_n \omega_n(M)$. For 
 every embedding $\sigma: k\hookrightarrow C$ into an algebraic closure $C$ of
$k$, there is a Betti realisation  $M_\sigma$, which is a finite dimensional vector space over $\Q$ equipped with a weight filtration such  as above \cite{DG}, \S2.11.
There is a comparison isomorphism which respects the weight filtrations
\begin{equation}\label{compsdr}
\comp_{\sigma, DR}: M_{DR} \otimes_{k,\sigma} C \overset{\sim}{\To}
M_\sigma \otimes_{\Q}  C\ , \end{equation} and is functorial with
respect to $\sigma$.  For each $\sigma:k \hookrightarrow C$, the data
\begin{equation}\label{Hsigma}
H_{\sigma}=(M_{DR}, M_\sigma,
\comp_{\sigma,DR})\end{equation}
is called a mixed Hodge-Tate structure   \cite{DG}, \S2.13.
  The Hodge realisation
functor:
$$M\mapsto  (M_{DR}, M_\sigma,
\comp_{\sigma,DR})_{\sigma:k\hookrightarrow C}$$
 to the category of systems of  mixed Hodge-Tate structures is fully faithful, although
we shall not require this fact. Note that $\omega_n(M)\subset W_{-2n} M_{DR}$ can be retrieved from the data $(\ref{Hsigma})$ (\emph{loc. cit.}).

\subsubsection{Framed objects in $\MT(k)$}
 Let $M\in \MT(k)$,  and let $n\geq 0$.
\begin{defn}    An \emph{$n$-framing}
of $M$ \cite{BGSV, Go3} consists of non-zero
morphisms:
\begin{eqnarray}
v_0& \in& \omega_0(M)=\Hom( \Q(0), \gr_0^W M )\ , \nonumber \\
f_n& \in &  \omega_n(M)^\vee=\Hom(\gr_{-2n}^W M, \Q(n) )\ .
\nonumber
\end{eqnarray}
A morphism from $(M,v_0,f_n)$ to $(M,v_0', f_n')$ is  a morphism
$\phi: M \rightarrow M'$
such that $\phi  (v_0)=v_0'$ and $f'_n  \circ \phi= f_n$. Morphisms  generate an equivalence relation on the set of $n$-framed objects. The
equivalence class of $(M,v_0,f_n)$ is  written
 $[M,v_0,f_n]$. 
\end{defn}

Let $\Am_n(k)$ denote the set of  equivalence classes of 
$n$-framed objects in $\MT(k)$. One  shows  that   $\Am_n(k)$
is a $\Q$-vector space with respect to the addition rule:
\begin{equation} \label{sumofframedmotives}
[M, v_0, f_n] + [M',v_0', f_n'] = [M \oplus M', v_0\oplus v_0'  ,
f_n+f_n']\ ,
\end{equation}
 and scalar multiplication $\alpha  [M, v_0, f_n]
=[M,\alpha  v_0, f_n] = [M, v_0, \alpha  f_n]$ for all $\alpha \in
\Q^\times$.
 The zero element is given by the  equivalence class of $\Q(0)\oplus \Q(n)$ with trivial
framings, and    $\Am_0(k)\cong \Q$.  Consider the graded
$\Q$-vector space:
\begin{equation}
\Am(k) = \bigoplus_{n\geq 0} \Am_n(k)\ .
\end{equation}
It is equipped with a coproduct $\Delta: \Am(k)\rightarrow \Am(k)\otimes_{\Q} \Am(k)$, whose components
$$\Delta_{r,n-r}: \Am_n(k)\rightarrow \Am_{r}(k)\otimes_{\Q} \Am_{n-r}(k)$$
 can be computed  as follows. For  $M \in \MT(k)$,  let 
 $\{e_i\}_{1\leq i \leq N}$ be any basis of $\omega_r(M)$,  and let
$\{e_i^{\vee}\}_{1\leq i\leq N}$ denote the dual basis in
$\omega_r(M)^\vee$. Define
\begin{equation} \label{coproddef1}
\Delta_{r,n-r} [M,v_0,f_n] = \sum_{i=1}^N  [M,v_0, e_i^{\vee}]
\otimes [M,e_i, f_n](-r)  \ ,
\end{equation}
where $[M,e_i, f_n](-r)$ is the Tate-twisted object $M(-r)$ with
corresponding framings. Let $\Delta_n = \bigoplus_{0\leq r\leq n} \Delta_{r,n-r}$, and set $\Delta=\bigoplus_{n \geq 0} \Delta_n$. One verifies that the above constructions are well-defined and that   
 $\Am(k)$ can  be made into a  graded commutative Hopf algebra. 
Define the reduced
coproduct $\widetilde{\Delta}$ on $\Am(k)$ by 
$\widetilde{\Delta}(X) = \Delta(X) - 1\otimes X - X\otimes 1$. Its
kernel is  the set of primitive elements in
$\Am(k)$.
\begin{prop}  (see e.g., \cite{Go1}) \label{corkerdelta}  There is an isomorphism:
\begin{equation} \label{kerdeltaisext}
\Ext^1_{\MT(k)}(\Q(0),\Q(n)) \cong \ker \Big(\widetilde{\Delta}_n :
\Am_n(k) \To \bigoplus_{1\leq r\leq n-1} \Am_r(k)\otimes_\Q
\Am_{n-r}(k)\Big) \ .\end{equation}
\end{prop}
This gives a  strategy  for defining elements in rational  algebraic $K$-theory via $(\ref{extgroups})$.
The essential idea  \cite{BD} is to  construct a  framed mixed Tate
motive out  of simple algebraic varieties in such
a way that the reduced coproduct vanishes.

\begin{rem} In this paper, the mixed Tate motives that we construct  are cohomological, and will typically have framings in weights $0$ and $2n$. They will be systematically be Tate twisted by $(n)$
to bring their weights to $-2n$, and $0$ respectively, in order to fit with the standard formalism above.
\end{rem}

\subsubsection{Real periods}
 Let $M\in \MT(k)$, and $\sigma:k\hookrightarrow \C$. Consider  the isomorphism:
$$P_\sigma:\omega(M)\otimes_{\Q} \C  \overset{\sim}{\To} M_{DR}\otimes_{k,\sigma} \C
\overset{\comp_{\sigma,DR}}{\To} M_\sigma \otimes_{k,\sigma} \C\ ,
$$
 and let
$P_\sigma^*=(P_\sigma^\vee)^{-1}$  be the inverse dual map from
$\omega(M)^\vee\otimes_{\Q} \C$ to $ M_\sigma^\vee\otimes_{k, \sigma} \C$.
  Denote the natural pairing
$\big(M_\sigma\otimes_{k, \sigma} \C\big) \otimes   \big(M_\sigma\otimes_{k, \sigma} \C\big)^\vee\rightarrow \C$ by $\langle \cdot, \cdot \rangle$.

\begin{defn} \label{defnrealperiod}
Let $M\in \MT(k)$, with $n$-framings $v_0\in \omega(M)$, $f_n \in \omega(M)^{\vee}$. For every $\sigma:k\hookrightarrow \C$ its real period  (\cite{Go1}, \S4)
is defined by
$$ \langle \Image \big((2i\pi)^{-n}P_\sigma (v_0\otimes 1)\big), \Real (P_\sigma^* (f_n\otimes 1)) \rangle  \in \R\ .$$
This only depends on the equivalence class $[M, v_0, f_n]$, yielding a map:
\begin{equation} \label{realperdef}
\Rp_\sigma: \Am_n(k) \To \R \ ,
\end{equation} 
 which is called the real period. Note that its normalization varies in the literature.
\end{defn}

  To calculate the real period,  choose a graded $k$-basis
of $M_{DR}$ and  a  $\Q$-basis of $M_\sigma$ which is compatible with the weight
filtration. The matrix $\comp_{\sigma,DR}$ in these bases can 
be computed  from  the integration pairing 
$$M_{DR}\otimes_{\Q} M^{\vee}_{\sigma}\rightarrow \C\ .$$
 We can assume
 that $\comp_{\sigma,DR}$ on  $
 \gr^W_{-2m} M_{DR}\otimes_{k,\sigma} \C \overset{\sim}{\rightarrow}\gr^W_{-2m} M_\sigma\otimes_{k,\sigma} \C$ is
$(2\pi i)^m$ times the identity for all $m$. Then our $k$-basis of $M_{DR}$ is a $\Q$-basis of $\omega(M)$ (\cite{DG}, (2.11.3)) and our matrix
represents  $P_\sigma$.  It is well-defined up to multiplication by a triangular unipotent matrix with entries in $\Q$, representing a change of filtered $\Q$-basis of $M_{\sigma}$.
 By $(\ref{kerdeltaisext})$, we have  $\Ext^1_{\MT(k)} (\Q(0), \Q(n))= \ker \widetilde{\Delta}_n\subset \Am_n(k), $ and hence a map:
 \begin{equation} \label{Rsigmaonext}
 R_\sigma: \Ext^1_{\MT(k)} (\Q(0), \Q(n)) \To \R\ .
\end{equation}
To compute this map, we can represent an  element  $\xi\in \Ext^1_{\MT(k)} (\Q(0), \Q(n))$  by an extension  $0 \rightarrow \Q(n) \rightarrow M \rightarrow \Q(0)\rightarrow 0$.  For every $\sigma :k \hookrightarrow \C$, the prescription above yields a period
 matrix for $M$ which is of the form
$$ P_\sigma(M)=
 \left(
   \begin{array}{cc}
     1 & 0 \\
     \alpha & (2\pi i)^n  \\
   \end{array}
 \right) \hbox{ for some } \alpha \in \C.
$$ 
Therefore its real period, for the obvious framings,  is  $R_{\sigma}(\xi)=\Image ({\alpha \over (2\pi i)^n}).$

\subsubsection{Regulators}\label{subsectdetandhodge}
Let $k$ be a number field, and let  $\Sigma=\{\sigma_1,\ldots, \sigma_N\}$ denote any set of distinct  embeddings of $k$ into  $\C$.
Writing $k_i= \sigma_i(k)$, there is an isomorphism:
\begin{equation}\label{twistmap}
\rho  :\Am_n(k)\otimes_{\Q} \ldots \otimes_{\Q} \Am_n(k) \overset{\sim}{\To} \Am_n(k_1)\otimes_{\Q} \ldots \otimes_{\Q} \Am_n(k_N)\ .\end{equation}
Let $\Am_n(k)^{\otimes N}$ denote   the left-hand side of $(\ref{twistmap})$. The  symmetric group $\Sym_N$ acts upon it by 
permuting the factors.  Let $\Am_n(k)^{\otimes \Sigma}$ denote the right-hand side of $(\ref{twistmap})$, with the induced $\Sym_N$-action. If  $\xi_i \in \Am_n(k_i)$, for $1\leq i\leq N$,    $\Sym_N$   acts by:
$$ \pi(\xi_{1}\otimes
\ldots \otimes \xi_{N}) = \sigma_1
\sigma_{\pi(1)}^{-1}(\xi_{{\pi(1)}}) \otimes
\ldots \otimes \sigma_N \sigma_{\pi(N)}^{-1}(\xi_{{\pi(N)}})\quad \hbox{ for } \pi \in \Sym_N\ .$$
For any subspace $V\subset \Am_n(k)$ write $\bigwedge^N V\subset V^{\otimes N}$ for the subspace 
 of elements which are alternating with respect to the action of $\Sym_N$. Then by $(\ref{kerdeltaisext})$:
\begin{equation} \label{comdiagext}
\begin{array}{ccc}
 \textstyle{\bigwedge^N} \Ext^1_{\MT(k)}(\Q(0),\Q(n))  &  \subset   \textstyle{\bigwedge^N} \Am_n(k) & \subset  \Am_n(k)^{\otimes N}  \\ 
  \downarrow\!\! \rotatebox{90}{{$\sim$}} &  \downarrow \!\!\rotatebox{90}{{$\sim$}} & \downarrow \!\!\rotatebox{90}{{$\sim$}}  \\
\textstyle{\bigwedge^{\Sigma}} \Ext^1_{\MT(k)}(\Q(0),\Q(n))  &    \subset   \textstyle{\bigwedge^{\Sigma}} \Am_n(k) &   \subset  \Am_n(k)^{\otimes \Sigma} 
\end{array}
\end{equation}
where a superscript $\Sigma$ in the second line denotes the subspace of alternating elements with respect to the  action of $\Sym_N$ on $\Am_n(k)^{\otimes \Sigma} $.
Our construction of Dedekind zeta motives will naturally lie in the second line of $(\ref{comdiagext})$. 

\begin{defn} Suppose that $k$ has $r_1$ distinct real embeddings $\sigma_1,\ldots, \sigma_{r_1}$ and $2r_2$  distinct complex embeddings 
 $\sigma_{r_1+1},\ldots, \sigma_{r_1+r_2}, \overline{\sigma}_{r_1+1},\ldots, \overline{\sigma}_{r_1+r_2}$. Let
$$\Sigma_n=
\left\{
  \begin{array}{ll}
\{\sigma_1,\ldots,\sigma_{r_1+r_2}\} \    , & \hbox{ if } n \hbox{ is odd} \ , \\
  \{\sigma_{r_1+1},\ldots, \sigma_{r_1+r_2}\} \ , & \hbox{ if } n \hbox{ is even }\ .
  \end{array}
\right.
  $$
Let $n >1$. By $(\ref{extgroups})$ and $(\ref{Ktheorydimensions})$,  $\dim_\Q \big(\Ext^1_{\MT(k)}(\Q(0),\Q(n))\big) = |\Sigma_n|$.  In this situation $\textstyle{\bigwedge^{|\Sigma_n|}} \Ext^1_{\MT(k)}(\Q(0),\Q(n))$
 is a $\Q$-vector space of dimension 1. Let 
$$R_{\Sigma_n}= \prod_{\sigma \in \Sigma_n} R_{\sigma}:\Am_n(k)\otimes_{\Q} \ldots \otimes_{\Q} \Am_n(k) \To \R\ ,$$ be  the product of the real periods,  and 
define the  \emph{Hodge regulator  map} :
\begin{equation}\label{Hodgeregulatormap}
R_{\Sigma_n}: \textstyle{\det} \,\Ext^1_{\MT(k)}(\Q(0),\Q(n))  {\To} \R \qquad \hbox{ for } n>1\  .
\end{equation}
It factors through   $(\ref{comdiagext})$.
Its image is a  one-dimensional $\Q$-lattice in $\R$, whose covolume  is a well-defined element  $ R_n(k)\in \R^{\times}/\Q^\times\ .$
\end{defn}
The Hodge regulator $r_H: \Ext^1_{\MT(k)}(\Q(0),\Q(n))\rightarrow \R^{n_\pm}$ of  $(\ref{introhodge})$ is  the map $r_H= \bigoplus_{\sigma \in\Sigma_n} R_\sigma$. Its image
is a $\Q$-lattice whose covolume is $R_n(k) \mod \Q^\times$.

\subsection{Construction of the motive of  a product-hyperbolic
manifold} We begin by  considering  the motive of a single hyperbolic geodesic simplex;  first in the generic case following \cite{Go1} \S 3.3, and then with one  vertex at infinity.

\subsubsection{The motive of a finite geodesic simplex}  Let $m\geq 1$ and let  $\Delta\subset \Hyp^m$    be a  hyperbolic geodesic simplex with no vertices at infinity.
 In the Klein model for $\Hyp^m$, $\Delta$ is  a Euclidean simplex inside
the unit sphere $\partial \Hyp^m\subset \R^m$.
This data  can be represented by a smooth quadric $Q\subset \Pro^{m}$, and a set of  hyperplanes $\{L_0,\ldots,
L_m\}\subset \Pro^{m} $, defined over $\R$, such that  for some affine open $\A^m \subset \Pro^m$, $Q(\R)\cap \A^m$ is the unit sphere in $\R^m$, and the boundary faces of $\Delta$ 
are contained in $\bigcup_i L_i(\R)\cap \A^m$.

Let $L=\bigcup_i L_i$.  If $Q, L$ are defined over $\overline{\Q}$,  and since $L \cup Q$ is a normal crossing divisor, this data directly defines a mixed Tate motive
\begin{equation}\label{finitemotive}
h(\Delta)=H^m(\Pro^m\backslash Q, L\backslash (L\cap Q)) \in \MT(\overline{\Q})\ ,
\end{equation}
which only depends on $\Delta$.   To determine the structure of $h(\Delta)$, observe that  for any smooth quadric $Q\subset \Pro^m$, 
$$H^m(\Pro^m \backslash Q) = \left\{
                             \begin{array}{ll}
                               \Q(-n)\ , & \hbox{ if }  m=2n-1\  , \\
                                 0 \ , & \hbox{ if  } m \hbox{ is evenÊ}\ .
                             \end{array}
                           \right.$$
Whenever $m=2n-1$ is odd, one verifies that $\gr^W_{2n} h(\Delta) \cong H^{m}(\Pro^{m}\backslash Q) $, and hence we obtain a class (see \S\ref{sectFramings}  below) which we denote by
\begin{equation}\label{canclassomegaq}
\omega_Q : \Q(-n) \overset{\sim}{\To} \gr^W_{2n} h(\Delta)\ .\end{equation}
Write $L_I= \bigcap_{i\in I} L_i$ for any subset
$I\subset\{0,\ldots, m\}$, and  let  $Q_I=Q\cap L_I$.  
Every  face $F$ of  $\Delta$ is equal to  $\Delta\cap L_I$  for some such $I$ and write $Q_F$ for $Q_I$. Since $F$ is a geodesic simplex in $\Hyp^{m-|I|}$, it defines a motive
$h(F)$,  corresponding to the quadric $Q_I \subset L_I$ relative to the hyperplanes $L_I \cap L_j$, for $j \notin I$.
For each face $F$ of odd dimension $2i-1$, the inclusion $ F\hookrightarrow \Delta$ induces a map $i_F:h(F)\rightarrow h(\Delta)$ and we set
\begin{equation} \label{eFfinitedef}
e_F= i_F\circ \omega_{Q_F} : \Q(-i)\To \gr^W_{2i} h(\Delta)\ .\end{equation}
\begin{prop}  Let $m=2n-1\geq 1$. Then $\gr^W_0 h(\Delta)=\Q(0)$ and 
\begin{equation} \label{finitemotivegradedpieces}
\gr^W_{2(n-r)} h(\Delta)=
 \bigoplus_{|I|=2r}\Q(r-n) \quad \hbox{ for  } \quad 0\leq r< n \ .
\end{equation}
A basis is given by the image of    $1 \in \Q(-i)$ under the map  $e_F$, as $F$ ranges over  all odd-dimensional faces of $\Delta$.
\end{prop}
This result will  follow from our proof of proposition  $\ref{propmotiveofsimplex}$ below.

\subsubsection{The motive of a simplex with a vertex at infinity} \label{secmotiveatinf}
Let $\Delta$ be  a hyperbolic  geodesic simplex  with a single vertex $x$ on the boundary $\partial \Hyp^m$. As above, 
$\Delta$ defines a set of hyperplanes $L_0,\ldots, L_m$ and a smooth quadric $Q$ in $\Pro^m$, which do not cross normally at $x$.
Number the hyperplanes so  that $L_1,\ldots, L_m$ intersect at $x$, and $L_0$ is the hyperplane which does not contain $x$.
Let $\widetilde{\Pro}^m$ denote the blow-up of $\Pro^m$ at $x$, let $\widetilde{L}_{-1}$  denote the exceptional divisor, and let
 $\widetilde{Q}$, $\widetilde{L}_i$ be the strict transforms of $Q, L_i$ respectively.
  When $Q,L_i$ are defined over $\overline{\Q}$,    we  define a mixed Tate motive:
\begin{equation}\label{infinitemotive}
h(\Delta)=H^m(\widetilde{\Pro}^m\backslash \widetilde{Q}, \widetilde{L}\backslash (\widetilde{L}\cap \widetilde{Q})) \in \MT(\overline{\Q})\ .
\end{equation}
It  has an identical structure to the motive of a finite simplex  $(\ref{finitemotive})$, except in graded weight $2$.
This corresponds to the fact that the one-dimensional faces of $\Delta$ which meet  $x$ have infinite length.
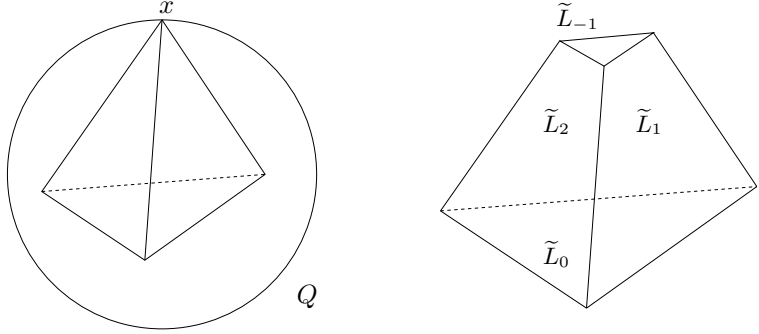
\begin{figure}[h!]
  \begin{center}
    \epsfxsize=10.0cm \epsfbox{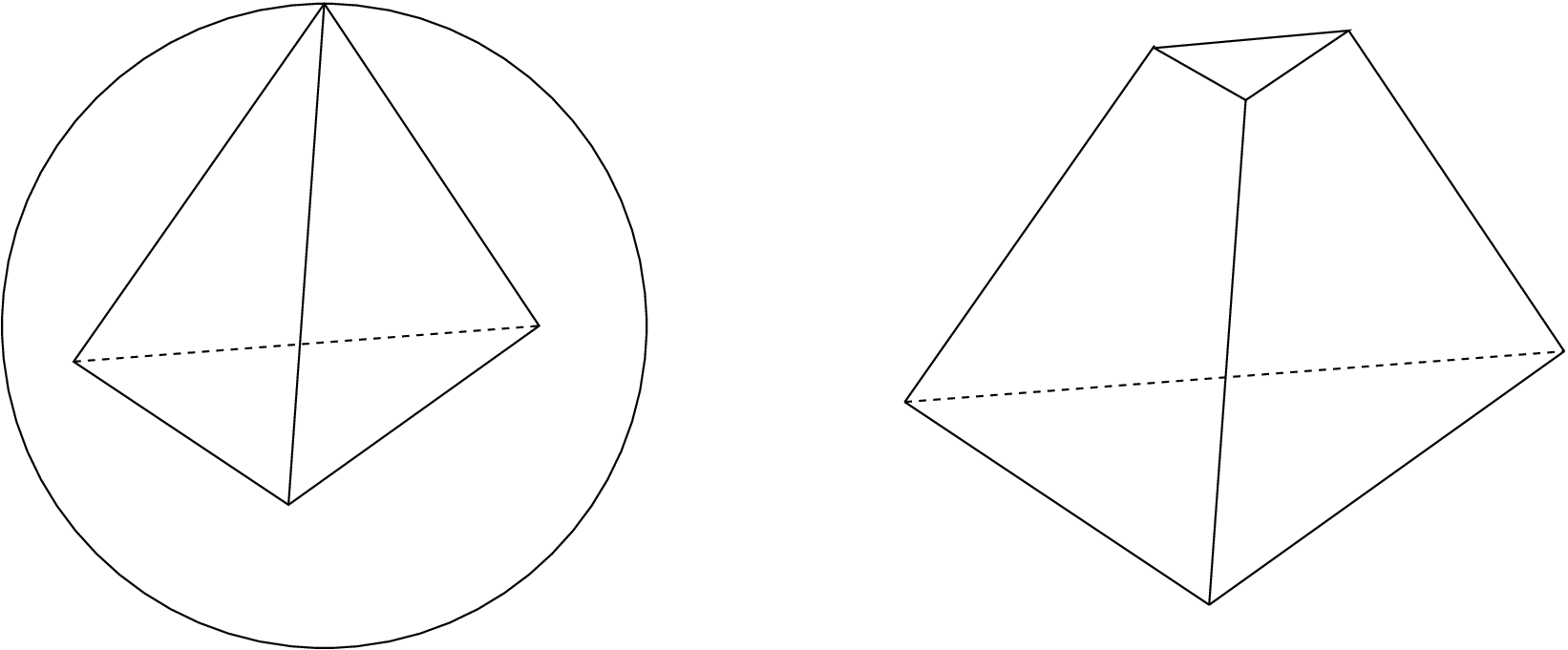}
  \label{Blownupsimplex}
\put(-230,120){$x$}
\put(-178,10){$Q$} \put(-80,115){{\small $\widetilde{L}_{-1}$}} \put(-50,75){\small{$\widetilde{L}_1$}}
\put(-85,75){\small{$\widetilde{L}_2$}}
\put(-85,25){\small{$\widetilde{L}_0$}}
 \caption{A hyperbolic 3-simplex $\Delta$ with one vertex  at infinity $x$ in the Klein model. After blowing up  $x$, the exceptional divisor
$\widetilde{L}_{-1}$ meets the faces of  its inverse image $\widetilde{\Delta}$ in a Euclidean triangle. }
  \end{center}
\end{figure}
We construct  a basis of  $\gr^W_\bullet h(\Delta)$ in the case $m=2n-1$ is odd  as follows. As in the  case of a finite simplex, there is a map
\begin{equation}
\omega_{\widetilde{Q}}: \Q(-n) \overset{\sim}{\To} \gr^W_{2n} h(\Delta) \cong  H^{2n-1}(\widetilde{\Pro}^{2n-1}\backslash \widetilde{Q})\ .
\end{equation}
Similarly, for each face $F$ of odd dimension $2i-1\geq 3$, the inclusion of the face $F\hookrightarrow \Delta$ defines a map
$i_{\widetilde{F}}:h(F) \rightarrow h(\Delta)$ and gives rise to  a map\footnote{Here, and subsequently, the tildes in the notation $e_{\widetilde{F}}$ and $i_{\widetilde{F}}$ are not strictly necessary but are there as a reminder that we are in the case when a vertex at infinity has been blown up.}
:
\begin{equation}\label{eFinfinitedef}
e_{\widetilde{F}}= i_{\widetilde{F}} \circ \omega_{\widetilde{Q}_{\widetilde{F}}}:\Q(-i)\overset{\sim}{\To} \gr^W_{2i} h(F) \To \gr^W_{2i} h(\Delta)\ .
\end{equation}
 The same holds for one-dimensional faces which do not contain $x$.   
  However, for every one-dimensional face $F$  containing  $x$, we have $\gr^W_{2} h(F)=0$ since its strict transform meets $\widetilde{Q}$ in a single point, and $H^1(\Pro^1\backslash \{1 \hbox{ point}\})=0$.

What happens instead is the following.
Let $G$ be a $2$-dimensional face of $\Delta$ which meets $x$. Its strict transform  corresponds to
the complement of the blow-up  of a smooth quadric in the projective plane
$\widetilde{\Pro}^2\backslash \widetilde{Q}_G$. There is  a   class
\begin{equation}\label{etaG}
\eta_{\widetilde{Q}_G}: \Q(-1) \overset{\sim}{\To} H^2(\widetilde{\Pro}^2 \backslash \widetilde{Q}_G)\cong \gr^W_2 h(G) \ ,\end{equation}
which we shall define below. As previously, the inclusion of the face $G\hookrightarrow \Delta$ induces a map
$i_{\widetilde{G}}: h(G) \rightarrow h(\Delta)$, and we define:
\begin{equation}\label{alphaGdef}
\alpha_{\widetilde{G}} = i_{\widetilde{G}} \circ \eta_{\widetilde{Q}_G}: \Q(-1)\To \gr^W_2 h(\Delta)\ .  \end{equation}
Now  consider the set of hyperplanes  $L_1$,\ldots,$L_m$ which contain $x$.
Their strict transforms  $\widetilde{L}_{1},\ldots, \widetilde{L}_{m}$ intersect $\widetilde{L}_{-1}$ in an  $m-1$ simplex $\Delta_\infty$. Let
$F_k(\Delta_{\infty})$ denote the set of faces of $\Delta_\infty$ of dimension $k$. From $(\ref{alphaGdef})$ we deduce a map
\begin{eqnarray} \label{qminus1map}
\alpha: \Q(-1)^{F_1(\Delta_{\infty})} \To \gr^W_2 h(\Delta)
\end{eqnarray}
since every $1$-dimensional face of $\Delta_\infty$ corresponds to a 2-dimensional face of $\Delta$ containing $x$.
Choose an orientation on $\Delta_{\infty}$ and set
\begin{equation} \label{Vxdef}
V_x= \mathrm{coker} ( \Q^{F_2(\Delta_{\infty})}\overset{\partial}{\To} \Q^{F_1(\Delta_{\infty})})
\end{equation} 
where $\partial$ is the boundary map.  We shall show that $(\ref{qminus1map})$ factors through
$$\alpha: V_x\otimes_{\Q}\Q(-1) \To \gr^W_2 h(\Delta)Ê\ .$$
Note that $V_x$ is of dimension $m-1$ and  isomorphic to the space
\begin{equation} \label{Vxasker}
V_x\cong \ker (\Q^{F_0(\Delta_{\infty})} \To \Q)\ ,\end{equation}
spanned by linear combinations of $1$-dimensional faces of $\Delta$ which meet $x$ such that the sum of coefficients is $0$. 
This corresponds to the fact
that there is a relation between the internal angles of the Euclidean simplex at infinity $\Delta_\infty$ (see \S\ref{secquotmotives}).
\begin{prop}  \label{propmotiveofsimplex}
Let $m=2n-1\geq 1$ be odd. Then  $\gr^W_0 h(\Delta) \cong  \Q(0)$ and 
\begin{eqnarray}\label{polymotivegradedpieces}
\gr^W_{2(n-r)} h(\Delta)=\bigoplus_{|I|=2r}\Q(r-n) \quad \hbox{for} \quad 0\leq r \leq n-2\ 
\end{eqnarray}
in weights $\geq 4$. A  basis is given  by  the images of $1$ under  $e_{\widetilde{F}}$ for all odd-dimensional faces $F$ of $\Delta$ of dimension $\geq 3$. 
In graded weight $2$, we have
\begin{equation}\label{infmotivegraded2piece}
\gr^W_2 h(\Delta) \cong  \Q(-1)^{m(m+1)/2 -1 }\ ,\end{equation}
which is spanned by the images of $1$ under   $e_{\widetilde{F}}$, for the $m(m-1)/2$ one-dimensional faces $F$ of $\Delta$ which do not contain $x$,  
and  the image of $1$ under $\alpha_{\widetilde{G}}$ for all two-dimensional faces $G$ of $\Delta$ which contain $x$, modulo boundaries of three-dimensional faces of $\Delta$ which meet $x$. 
\end{prop}

Before proving proposition $\ref{propmotiveofsimplex}$ it is useful to state the following lemma.

\begin{lem} Let $Q\subset \Pro^m$ be  a smooth quadric,  and $x$ a point in  $Q$. Let $\widetilde{\Pro}^m$ denote the blow-up of $\Pro^m$ at $x$, and $\widetilde{Q}$ the strict transform of $Q$. Then 
\begin{equation} \label{factaboutquadrics}
H^i(\widetilde{\Pro}^m\backslash \widetilde{Q}) = \left\{
                              \begin{array}{ll}
                                \Q(-n)\ , & \hbox{ if  } i=m \hbox{ is odd and equal to } 2n-1\ ;\\
\Q(-1)\ , & \hbox{ if  } i=2\ ; \\
\Q(0)\ , & \hbox{ if  } i=0\ ; \\
                                0\ , & \hbox{ otherwise }\ .
                              \end{array}
                            \right.
\end{equation}

\end{lem}
\begin{proof}
Recall that the cohomology of a smooth quadric $Q$ of dimension $\ell$ vanishes in odd degrees, and satisfies
$H^{2i}(Q)= \Q(-i)$ for $1\leq i\leq \ell$, unless $\ell=2k$ is even, in which case
$H^{2k}(Q)=\Q(-k)\oplus \Q(-k)$ in middle degree.  When $\ell>1$, the strict transform $\widetilde{Q}$ is just the blow-up of  $Q$ in the point $x$. From  the Gysin sequence:
$$\ldots \To H^{i-2}(\widetilde{Q})(-1) \To H^i(\widetilde{\Pro}^m) \To H^i(\widetilde{\Pro}^m\backslash \widetilde{Q}) \To H^{i-1}(\widetilde{Q})(-1) \To \ldots\ ,$$
and the formula for the cohomology of a blow-up of a smooth variety in a point, we obtain the statement.
\end{proof}

\begin{proofofp}$\ref{propmotiveofsimplex}$. 
The motive $(\ref{infinitemotive})$ can be computed from the hypercohomology of $\widetilde{\Pro}^m \backslash \widetilde{Q}$ in   the complex of sheaves
\begin{equation}\label{compofsheaves}
\Q_{\widetilde{\Pro}^m \backslash \widetilde{Q}} \rightarrow \bigoplus_{|I|=1} \Q_{\widetilde{L}_I \backslash \widetilde{Q}_I} \rightarrow \ldots \rightarrow  \bigoplus_{|I|=m} \Q_{\widetilde{L}_I \backslash \widetilde{Q}_I} 
\end{equation} 
 where $I\subset \{-1,0,\ldots, m\}$, and $\Q_Y$ denotes the constant sheaf $\Q$ on $Y\subset \widetilde{\Pro}^m \backslash \widetilde{Q}$, extended by zero outside $Y$, and the morphisms are the restriction maps.
It  defines a spectral sequence  with \begin{equation} \label{E1specsequence}
E_1^{p,q} = \bigoplus_{|I|=p} H^q(\widetilde{L}_I\backslash \widetilde{Q}_I)\quad \hbox{ for
}  p\geq 1\ , \quad\hbox{ and} \quad E_1^{0,q}=H^q(\widetilde{\Pro}^m\backslash \widetilde{Q})\ ,
\end{equation}
which converges to
$H^{p+q}(   \widetilde{\Pro}^m \backslash \widetilde{Q}, \widetilde{L} \backslash (\widetilde{L} \cap  \widetilde{Q} ))$.
If $-1 \in I$ then  $\widetilde{L}_{I}\backslash \widetilde{Q}_{I}$ is isomorphic to  affine space,  and so $H^q(\widetilde{L}_I\backslash \widetilde{Q}_I)=0$ for all $q\geq 1$. If $0 \in I$ then  $\widetilde{L}_{I}\backslash \widetilde{Q}_{I}\cong L_I \backslash Q_I $ is the complement of
a smooth quadric in $\Pro^{m-|I|}$ and $H^q(L_I \backslash Q_I)$ vanishes when $q>0$, unless $|I|=2r$ is even,  and $q=m-2r$.
For all other $I$, $H^q(\widetilde{L}_I\backslash \widetilde{Q}_I)$ is given by $(\ref{factaboutquadrics})$.
The spectral sequence  $(\ref{E1specsequence})$ degenerates at $E_2$, and one  deduces that\begin{equation} \label{inproofgr}
\gr^W_{2(n-r)} h(\Delta) = \bigoplus_{-1\notin I, |I|=2r} H^{2(n-r)-1}(\widetilde{L}_I\backslash \widetilde{Q}_I)\cong  \bigoplus_{-1\notin I, |I|=2r}  \Q(r-n)\ , 
\end{equation}
for $0\leq r\leq n-2$, and  where we write $\widetilde{L}_{\emptyset} \backslash \widetilde{Q}_{\emptyset}$ for $\widetilde{\Pro}^m \backslash \widetilde{Q}$.
Now by   $(\ref{factaboutquadrics})$
$$E_1^{p,2} = \bigoplus_{I\subset \{1,\ldots, m\},\, |I|=p} \Q(-1)\cong \Q(-1)^{\binom{m}{p}} \qquad \hbox{ for } 0\leq p\leq m-2\ . $$
 The complex $0\rightarrow E_1^{0,2} \rightarrow E_1^{1,2}\rightarrow \ldots
\rightarrow E_1^{m-2,2}$ is precisely the simplicial complex corresponding to $\Delta_{\infty}$:
$$0\rightarrow \Q^{F_{m-1}(\Delta_{\infty})}\rightarrow \ldots \rightarrow \Q^{F_2(\Delta_{\infty})} \rightarrow \Q^{F_1(\Delta_{\infty})}$$
tensored with $\Q(-1)$, which is exact except in the last position. By $(\ref{Vxdef})$,
$$E_2^{m-2,2} \cong  V_x\otimes_{\Q} \Q(-1)\cong  \Q(-1)^{m-1}\ ,$$
and $E_2^{p,2}=0$ for $p<m-2$. 
Finally, we consider the one-dimensional edges. If $|I|=m-1$, then  $\widetilde{L}_I\cong \Pro^1$. It  meets $\widetilde{Q}_I$ in
two points if $0\in I$, and in exactly one point if $0\notin I$. In the latter case $H^1(\widetilde{L}_I\backslash \widetilde{Q}_I)=0$, so we have
$$E_2^{m-1,1}=E_1^{m-1,1} = \bigoplus_{0\in I,   \, |I|=m-1}  H^1(\widetilde{L}_I\backslash \widetilde{Q}_I)= \Q(-1)^{m(m-1)/2}\ .$$
 The total contributions in graded weight 2 are therefore:
$$\gr^W_{2} h(\Delta) \cong E_2^{m-2,2}\oplus E_2^{m-1,1} \cong \Q(-1)^{m(m+1)/2 -1 }\ .$$
Finally,  $E^{p,0}_1= \bigoplus_{|I|=m-p} \Q(0)$ and  $\ldots\rightarrow  E^{0,p}_1 \rightarrow E^{0,p+1}_1 \rightarrow \ldots$ is  the simplicial complex of the set of hyperplanes $\widetilde{L}$, which has the homology of a sphere. Thus
 \begin{equation} \label{inproofgrw0}
 \gr^W_0 h(\Delta) = E_2^{m,0} = 
\mathrm{coker} \big( \bigoplus_{|I|=m-1} \Q(0) \rightarrow
\bigoplus_{|I|=m} \Q(0) \big) \cong \Q(0) \ .
\end{equation} 
\end{proofofp}

\subsubsection{Framings} \label{sectFramings}  In both cases (whether $\Delta$ has a vertex at infinity or not), we have
$$\gr^W_{2n} h(\Delta) \cong H^{2n-1}(\Pro^{2n-1} \backslash Q) \cong \Q(-n) \quad \hbox{ and } \quad \gr^W_0 h(\Delta)\cong \Q(0)\ , $$
 by $(\ref{inproofgr})$ and  $(\ref{inproofgrw0})$.
 In the  first isomorphism, a generator is given by an element $\pm (1,-1)$ in  $H^{2n}(Q)(-1)\cong \Q(-n)\oplus \Q(-n)$, and is well-defined up to a sign.  In the second isomorphism, a generator  is given by the set of vertices of  $L$ (or the vertices of the  total transform  polytope $\widetilde{L}$  by $(\ref{inproofgrw0})$).
In the Hodge realisation, we can write these generators explicitly.  Let $x_0,\ldots, x_{2n-1}$ denote coordinates on $\Pro^{2n-1}$,
and let $q=\sum_{ij} a_{ij} x_ix_j$ be
a  quadratic form defining $Q$. Then, by $(4)$ in \cite{Go1}, 
\begin{equation} \label{omegaq}
\omega_Q  = \pm i^n \sqrt{\det q}{
\sum_{j=0}^{2n-1} (-1)^j x_j dx_0\wedge \ldots \wedge \widehat{dx}_j \wedge \ldots \wedge dx_{2n-1} \over q^n(x)}\ ,\end{equation}
defines a class $[\omega_Q]$ which generates $H_{DR}^{2n-1}(\Pro^{2n-1}\backslash Q)$. It does not depend on the choice of coordinates.
The framing on $\gr^W_0 h(\Delta)^\vee$ is given by the relative homology class of the real simplex $\Delta$ (or its inverse image $\widetilde{\Delta}$ in the case when it has a vertex at infinity), and  defines a class
$[\Delta] \in \gr^W_0 H_{B,2n-1}(\Pro^{2n-1}, L) \cong \gr^W_0 H^0_{B}( \Pro^{2n-1}, L)^\vee$ in both cases.
In practice, we shall consider many simplices $\Delta$ and one fixed quadric $Q$. Thus we fix  a sign of $\omega_Q$ at the outset, and all
signs of $[\Delta]$ are determined by an orientation of $\Delta$: its sign is  normalised so  that,
$$\vol(\Delta)=i^{1-n} \int_{\Delta} \omega_Q \in \R $$
is positive if $\Delta$ is positively oriented, and negative otherwise.

\begin{defn}
Let $\Delta$ be an oriented hyperbolic geodesic simplex in $\Hyp^{2n-1}$, with at most one vertex at infinity. Suppose it is defined over  $\overline{\Q}$. Let $h(\Delta)$ be defined by $(\ref{finitemotive})$ in the case where $\Delta$ is finite, and by $(\ref{infinitemotive})$ in the case where $\Delta$ has a vertex at infinity. Writing $(n)$ for the Tate twist, we set 
\begin{equation} \label{motsimplexdefn}
\mot(\Delta) = \big[h(\Delta), [\omega_Q], [\Delta]\big](n)
\in \Am_n(\overline{\Q})\ .\end{equation}
Note that it is convenient to  write the framings in terms of the de Rham and Betti classes, although they can be defined without reference to the 
Hodge realisation. 

In the case of an oriented hyperbolic geodesic triangle $\Delta$ in $\Hyp^2$ defined over $k$  with one vertex at infinity,  let $\eta_{\widetilde{Q}_{\Delta}}$ denote the class $(\ref{etaG})$,   $(\ref{factaboutquadrics})$ and set:
\begin{equation} \label{motsimplexdefn2}
\mot(\Delta) = \big[h(\Delta), [\eta_{\widetilde{Q}_{\Delta}}], [\widetilde{\Delta}]\big](1)
\in \Am_1( \overline{\Q})\ ,
\end{equation}
where $\widetilde{\Delta}$ is the  total inverse image of $\Delta$ in $\widetilde{\Pro}^2$.  

 \end{defn}

\begin{rem} \label{remmotvanishes} One can extend the definition of $h(\Delta)$ to the case where any number of vertices of $\Delta$ lie on the absolute in much the same way.
When the simplex $\Delta$ is degenerate, the resulting  framed object is equivalent to zero.
\end{rem}

\subsubsection{Definition of the framed motive of a product-hyperbolic manifold}
We first require the following subdivison lemma. Let $m$ be odd and $\geq 1$.

\begin{lem}\label{lemsubdiv} Let $x_0,\ldots, x_m$ be  distinct points in  $\overline{\Hyp}^{m}_{\overline{\Q}}$ in general position except that some may lie on the absolute, and let
$\Delta(x_0,\ldots, x_m)$ denote the geodesic simplex whose vertices are given by the $x_i$. Given any finite point $y\in \Hyp^m_{\overline{\Q}}$,
\begin{equation}\label{subdivide}
\mot(\Delta(x_0,\ldots, x_m)) = \sum_{i=0}^m \mot(\Delta(x_0,\ldots,x_{i-1}, y, x_{i+1},\ldots, x_m))\ .\end{equation}
\end{lem}

\begin{proof}
Let $J$ (resp. $I$) be the set of $m$-element subsets of  $\{x_0,\ldots, x_m\}$ (resp. $\{x_0,\ldots, x_m,y\}$), and let $L_{i}$ for $i\in I$  denote the hyperplane passing through $m$ points. Let $Q$ denote a smooth quadric in $\Pro^{m}$ corresponding to the absolute
of $\Hyp^m$. Let $\widetilde{\Pro}^m$ denote the blow-up of $\Pro^m$ 
at all intersections of hyperplanes which do not cross normally in increasing order of dimension so that  the strict transforms $\widetilde{L}_i$ of $L_i$, the strict transform  $\widetilde{Q}$  of $Q$,  and the exceptional loci $L_h$ for $h$ in some indexing set $H$,  are normal crossing.
  The hypercohomology of the complex of sheaves
\begin{equation}\label{pfcmplx}
\Q_{\widetilde{\Pro}^m\backslash \widetilde{Q}} \rightarrow \bigoplus_{i\in I\cup H} \Q_{\widetilde{L}_i \backslash (\widetilde{L}_i \cap \widetilde{Q})}
 \rightarrow \bigoplus_{i,j\in I\cup H} \Q_{\widetilde{L}_{ij}  \backslash (\widetilde{L}_{ij} \cap \widetilde{Q})} \rightarrow \ldots \end{equation}
defines a mixed Tate motive.  There are $m+1$ similar complexes defined relative to the vertices $\{x_0,\ldots, x_{i-1},y,x_{i+1},\ldots, x_m\}$ for each $i$,  which map 
 to $(\ref{pfcmplx})$. Via these maps, there are unique framings on the motive of $(\ref{pfcmplx})$  which make it equivalent to $\sum_{i=0}^m \mot(\Delta(x_0,\ldots,x_{i-1}, y, x_{i+1},\ldots, x_m))$. 
Likewise,  the complex
\begin{equation}\label{pfcmplx2}
\Q_{\widetilde{\Pro}^m\backslash \widetilde{Q}} \rightarrow \bigoplus_{i\in J\cup H} \Q_{\widetilde{L}_i \backslash (\widetilde{L}_i \cap \widetilde{Q})}
 \rightarrow \bigoplus_{i,j\in J\cup H}  \Q_{\widetilde{L}_{ij}  \backslash (\widetilde{L}_{ij} \cap \widetilde{Q})} \rightarrow \ldots
\end{equation}
also defines a mixed Tate motive, which can be given unique framings making it equivalent to $\mot(\Delta(x_0,\ldots, x_m))$ (after blowing-down superfluous divisors).
The natural map  from $(\ref{pfcmplx2})$ to $(\ref{pfcmplx})$ is an equivalence of framed motives, and corresponds to the fact that
the relative homology class $[\Delta(x_0,\ldots, x_m))]$ is equal to the sum of the relative homology classes
 $\sum_{i=0}^m [\Delta(x_0,\ldots,x_{i-1}, y, x_{i+1},\ldots, x_m))]$.
\end{proof}

\begin{defn}\label{defnmotiveofM}
Let $M$ be a  finite-volume product-hyperbolic manifold modelled on $\X^\n=\Hyp^{2n_1-1}\times \ldots \times\Hyp^{2n_{\!N}-1}$ and
 defined over the fields $S=(k_1,\ldots, k_N)$.
By proposition $\ref{proptriang}$, $M$ admits a product-tiling with geodesic-product simplices
$\Delta^{(i)}_{1}\times \ldots \times \Delta^{(i)}_{N}, $ for $1\leq i\leq R$, 
 where $\Delta^{(i)}_j$ are oriented geodesic simplices in $\overline{\Hyp}^{n_j}$,  defined over $k_j$, with at most one vertex at infinity. 
 
  The (framed) motive of $M$ is  defined to be the element
\begin{equation}\label{eqndefnmotiveofM}
\mot(M) = \sum_{i=1}^R \mot(\Delta^{(i)}_1)\otimes \ldots \otimes \mot(\Delta^{(i)}_N)  \in \Am_{n_1}(\overline{\Q})\otimes_{\Q} \ldots \otimes_{\Q}\Am_{n_N}(\overline{\Q})\ .
\end{equation}
\end{defn}
It follows from  lemma $\ref{lemsubdiv}$ that $\mot(M)$ is well-defined, since one can construct a common subdivision of
any two distinct product-tilings of $M$.

\subsubsection{Vanishing of the reduced coproduct} \label{sectcoproductvanishes}
We prove  that $\mot(M)$ is a product of extensions of $\Q(0)$ by $\Q(n_i)$ by  showing that it lies in the kernel of the reduced coproduct.
Let  $\Delta$  be a geodesic hyperbolic simplex defined over $\overline{\Q}$ with at most one vertex at infinity.
 The main point is that the reduced coproduct of $\mot(\Delta)$ only depends on the boundary of $\Delta$. By definition $(\ref{coproddef1})$
 \begin{equation}
 \Delta_{n-r,r} \mot(\Delta) = \sum_i  [h(\Delta), [\omega_Q], e_i^{\vee}] (n) \otimes  [h(\Delta), e_i,[\Delta] ] (r)
 \end{equation}
 where $\{e_i\}$  is a basis for $\gr^W_{2r} h(\Delta)$.  By proposition $\ref{propmotiveofsimplex}$, such a basis is provided by 
  the classes $e_{\widetilde{F}}$, where $F$ ranges over all
faces of $\Delta$ of dimension $2r-1$, or, in the case when $\Delta$ has a vertex at infinity and $r=1$, by  linear combinations of classes $\alpha_{\widetilde{G}}$ where $G$ is a  face of $\Delta$ of dimension two. For any such face $F$, we have
\begin{equation}\label{motFdef}
\mot(F)  = [h(\Delta), e_i, [\Delta]](r) \ ,
\end{equation}
if $e_i$ is  $e_{\widetilde{F}}$ or $\alpha_{\widetilde{F}}$.
Thus  we can  write the reduced coproduct:
\begin{equation}\label{coprodoverF}
\widetilde{\Delta} (\mot(\Delta)) = \sum_F  c_F   \otimes \mot(F)  \ ,
\end{equation}
where the sum is over  the boundary faces of $\Delta$. We say that $c_F$ is the coefficient of the face $F$. It is
  a   framed  motive that we will compute in $\S\ref{secquotmotives}$.

\begin{thm} \label{thmMinext} Let $M$ be a product-hyperbolic manifold defined over the number fields $(k_1,\ldots, k_N)$ as  above. Then  $\mot(M)\in \Am_{n_1}(\overline{\Q})\otimes_{\Q}\ldots \otimes_{\Q} \Am_{n_r}(\overline{\Q})$ and defines an element in
$\Ext^1_{\MT(k_1)}(\Q(0),\Q(n_1)) \otimes_{\Q} \ldots \otimes_{\Q} \Ext^1_{\MT(k_N)}(\Q(0),\Q(n_N))\ .$
\end{thm}
\begin{proof} Let $\X^\n= \prod_{i=1}^N \Hyp^{2n_i-1}$, and
  $M=\X^{\n}/\Gamma$ (notations from \S2).
By proposition $\ref{proptriang}$, $M$ admits a  tiling with product simplices $\Delta^{(i)}_1\times \ldots \times \Delta^{(i)}_N$ for $1\leq i\leq R$
  defined over $(k_1,\ldots, k_N)$ with at most one vertex at infinity.  It suffices to show that  
  \begin{equation}\label{motMinkerpf}
  \mot(M)\in \ker\big(\id_1 \otimes \ldots \otimes \id_{j-1}\otimes \Delta_{r,n_j-r} \otimes \id_{j+1}\otimes \ldots \otimes \id_N  \big)\ ,
  \end{equation}
 for each $1\leq j\leq N$ and all $1\leq r\leq n_j-1$. This would imply that  
$$\mot(M) \in \ker \widetilde{\Delta}_{n_1} \otimes_{\Q} \ldots \otimes_{\Q} \ker \widetilde{\Delta}_{n_N} \ ,$$
and it follows  by corollary $\ref{corkerdelta}$ that 
$$\mot(M) \in \Ext^1_{\MT(\overline{\Q})}(\Q(0),\Q(n_1)) \otimes_{\Q} \ldots \otimes_{\Q} \Ext^1_{\MT(\overline{\Q})}(\Q(0),\Q(n_N))\ .$$
Since the tiling of $M$ is defined over $(k_1,\ldots, k_N)$, it is invariant under the natural action of 
$\Gal(\overline{k}_1/k_1)\times \ldots \times \Gal(\overline{k}_N/k_N)$ on the product-simplices. It follows from  Galois descent for $K$-theory, or more precisely, \cite{DG}, 
$(2.16.2)$, that  $\mot(M)$ lies in the subspace
$\Ext^1_{\MT(k_1)}(\Q(0),\Q(n_1)) \otimes_{\Q} \ldots \otimes_{\Q} \Ext^1_{\MT(k_N)}(\Q(0),\Q(n_N)).$

To prove $(\ref{motMinkerpf})$, let $1\leq r\leq n_j-1$ and let
$$V\subset \bigotimes_{i=1}^{j-1}\Am_{n_i}(\overline{\Q})    \otimes_{\Q} \big(\Am_r(\overline{\Q})\otimes_{\Q} \Am_{n_j-r}(\overline{\Q})\big) \otimes_{\Q} \bigotimes_{i=j+1}^{N}\Am_{n_i}(\overline{\Q}) $$ denote the $\Q$-vector space spanned by all the terms 
\begin{equation}\label{facetermsoftriang}
 \bigotimes_{i=1}^{j-1}\mot(\Delta^{(k)}_i)   \otimes ( c_F \otimes \mot(F)) \otimes  \bigotimes_{i=j+1}^{N}\mot(\Delta^{(k)}_i)  \end{equation}
which can occur in   $\id_1 \otimes \ldots \otimes \id_{j-1}\otimes \Delta_{r,n_j-r} \otimes \id_{j+1}\otimes \ldots \otimes \id_N \big(\mot(M)\big)$ by $(\ref{coprodoverF})$.
Since the tiling is finite, $V$ is finite dimensional  and only depends on the set of  faces of codimension $\geq 1$. 
Let us set
$$v = \id_1 \otimes \ldots \otimes \id_{j-1}\otimes \Delta_{r,n_j-r} \otimes \id_{j+1}\otimes \ldots \otimes \id_N (\mot(M)) \in V\ .$$
 We shall  show that $v$ is zero.  For simplicitly, let us assume that $M$ is compact and that all simplices $\Delta_j^{(i)}$ have no vertices at infinity. 
  The tiling of $M$  lifts to a $\Gamma$-equivariant tiling of  the whole of $\X^\n$.  Let 
  $D=\bigcup_{i=1}^R \Delta^{(i)}_1\times \ldots \times \Delta^{(i)}_N\subset \X^\n$ be a  lift of the  tiling of $M$.
   Now there exists an $s>0$, such that  for all  $r>\!>0$ sufficiently large, there are elements $\gamma_1,\ldots, \gamma_{N_r}\in \Gamma$ such that $$B_r\subset \bigcup_{1\leq i\leq N_r} \gamma_i(D)\subset B_{r+s} $$
where    $B_r\subset B_{r+s}\subset \X^{\n}$ are products of balls of radius $r$ and $r+s$, and the translates $\gamma_i(D)$ of $D$ are distinct and tile $B_r$.
Since $\gamma_i$ is an isometry,  $\mot(\bigcup_i \gamma_i(M)) = N_r\,\mot(M)$, and so
$$  \id_1 \otimes \ldots \otimes \id_{j-1}\otimes \Delta_{r,n_j-r} \otimes \id_{j+1}\otimes \ldots \otimes \id_N (\mot( \bigcup_i\gamma_i(D)))=N_r\, v$$
 where $N_r$ is of order $r^{\dim \X{^\n}}$, by volume considerations.  On the other hand, by  a version of lemma $\ref{lemsubdiv}$ (which easily generalizes to hold for  polytopes), and the fact that the reduced coproduct of a simplex only depends on its boundary, contributions from internal faces cancel  and we can  
rewrite
$$   \id_1 \otimes \ldots \otimes \id_{j-1}\otimes \Delta_{r,n_j-r} \otimes \id_{j+1}\otimes \ldots \otimes \id_N (\mot( \bigcup_i\gamma_i(D))) $$ 
  as a sum of terms in $V$ of the form  $(\ref{facetermsoftriang})$ where all faces  are contained in 
  $B_{r+s}\backslash B_{r}$.  
By volume considerations, the number of such faces is of order $r^{(\dim\X^\n-1)}$.  
Since $V$ is finite-dimensional,  choose a norm $||.||$ on $V$. The above argument shows that 
$||N_rv|| $ is of order $r^{(\dim\X^\n-1)}$. Since $N_r$ is of order $r^{\dim\X^\n}$, and $r$ is arbitrarily large, we have $||v||=0$ and hence $v=0$ as required.
The proof in the case when the simplices have vertices at infinity is entirely similar. 
  \end{proof}

\subsubsection{The symmetric group action}
It remains to show, in the case when $M$ is equivariant, that its framed motive $\mot(M)$ is a determinant.
Assume that all hyperbolic components are of equal odd dimension $2n-1$, {\it i.e.},  $n_1=\ldots =n_N=n$, and let $k$ be a totally real number
field with  a set of embeddings $\Sigma=\{\sigma_1,\ldots, \sigma_N\}$ into $\R$.\footnote{In the exceptional
 case $n=2$, if we identify $\SO^+(3,1)$ with $\PSL_2(\C)$ we can allow $k$ be to be any number field $L$ and $\sigma_i$ to be  complex  places of $L$.}
Let  $k_i=\sigma_i(k)$ and  suppose that
$M$ is equivariant with respect to $\Sigma$.

\begin{thm} \label{thmMmotiveisdet}
Let $M$ be  an equivariant product-hyperbolic manifold as above. Then, in the notation of $\S\ref{subsectdetandhodge}$,
the framed motive of $M$ is a determinant:
$$\mot(M) \in \textstyle{\bigwedge^{\Sigma}} \Am_n(k)\ .$$
\end{thm}

\begin{proof}
Let $\Delta^{(i)}_1\times\ldots \times \Delta^{(i)}_N$, for $1\leq i\leq R$ be a product-tiling for $M$, where $\Delta^{(i)}_j$ is defined over $k_j$, and has at most
one vertex at infinity. Since the cusps of $M$ are equivariant  the image of this tiling under the twisted action of the symmetric group $\Sym_N$
is another tiling for $\pi(M)=M$ (lemma $\ref{lemequivtile}$):
$$M=\bigcup_{i=1}^R (\sigma_1\sigma_{\pi(1)}^{-1}\Delta^{(i)}_{\pi(1)})\times\ldots \times (\sigma_N\sigma_{\pi(N)}^{-1}\Delta^{(i)}_N)\ , \ \hbox{ for all }\pi \in \Sym_N\ ,$$
It follows that $\mot(M)$ and $\mot(\pi(M))$ are equivalent up to a sign, determined by the action of $\pi$ on the framings.
First, $\pi$ preserves the framing in $\gr^W_{-2n} \mot(M)^\vee$ corresponding to the fundamental class of $M$,  because $M$ is equivariant.
Let $Q_i$ denote the quadric in $\Pro^{2n_i-1}$ which corresponds to $\partial \Hyp^{2n_i-1}$ , for $1\leq i\leq N$. 
The framing in $\gr^W_{0} \mot(M)$ corresponds to the volume form on $\X^n$, which is alternating:
$$\pi([\omega_{Q_1}\wedge \ldots \wedge \omega_{Q_N}]) = \varepsilon(\pi) [ \omega_{Q_1} \wedge \ldots \wedge \omega_{Q_N}]\ , \quad \hbox{for all } \pi \in \Sym_N\ ,$$
since each  $\omega_{Q_i}$ is of odd degree. Here, $\varepsilon$ is the sign of a permutation. Thus,
$$\mot(M) = \varepsilon(\pi)\,\pi (\mot(M)) \quad \hbox{ for all } \pi \in \Sym_N\ ,$$
and  $\mot(M)\in \Am_{n}(\sigma_1(k)) \otimes_{\Q} \ldots \otimes_{\Q} \Am_{n}(\sigma_N(k)) $ is a determinant  (\S\ref{subsectdetandhodge}).
\end{proof}

\subsubsection{Quotient motives and spherical angles} \label{secquotmotives}
Let  $m=2n-1$, and let $\Delta\subset \overline{\Hyp}^{m}_{\overline{\Q}}$ denote a  geodesic simplex with one vertex at infinity (the  case when $\Delta$ is finite is  simpler and left to the reader).
We wish to compute the coefficient motives  $c_F$ of any two or  odd-dimensional face $F$ of $\Delta$.    Using the notations of \S\ref{secmotiveatinf},  the motive $h(\Delta)$ is computed by the hypercohomology of the complex of sheaves $(\ref{compofsheaves})$:
\begin{equation}\label{blownupcomplex}
\Q_{\widetilde{\Pro}^m \backslash \widetilde{Q}} \rightarrow \bigoplus_{|I|=1} \Q_{\widetilde{L}_I \backslash \widetilde{Q}_I} \rightarrow \ldots \rightarrow  \bigoplus_{|I|=m} \Q_{\widetilde{L}_I \backslash \widetilde{Q}_I}  
\end{equation}
 where $I\subset \{-1,0,\ldots, m\}$.  First let  $F$  be a face of $\Delta$ of odd dimension $2r-1$. We assume that $F$ does not contain the vertex at infinity if $F$ is one-dimensional.  It   corresponds to a  hyperplane $L_J$ where $J\subset \{0,\ldots, m\}$, so $F=\Delta\cap L_J$,   and   $\mot(F)$ is given by  the hypercohomology of the following subcomplex of sheaves:
 \begin{equation}
 \Q_{\widetilde{L}_J \backslash \widetilde{Q}} \rightarrow \bigoplus_{|I|=|J|+1, J\subset I } \Q_{\widetilde{L}_I \backslash \widetilde{Q}_I} \rightarrow \ldots \rightarrow  \bigoplus_{|I|=m, J\subset I } \Q_{\widetilde{L}_I \backslash \widetilde{Q}_I}  
 \end{equation}
where  $I\subset \{-1,0,\ldots, m\}$.  Consider the  complex 
\begin{equation}
\label{coskeleton}
\Q_{\widetilde{\Pro}^{m}\backslash \widetilde{Q}} \rightarrow \bigoplus_{|I|=1,\, I \subset J} \Q_{\widetilde{L}_{I}\backslash \widetilde{Q}_{I}}
\rightarrow  \ldots  
\rightarrow  \Q_{\widetilde{L}_{J}\backslash \widetilde{Q}_{J}}\ .
\end{equation}
Its hypercohomology in degree $m$ is  $H^m(\widetilde{\Pro}^m \backslash \widetilde{Q}, \widetilde{L}^J \backslash \widetilde{Q}^J)$, where $L^J= \cup_{j\in J} L_j$, $Q^J = L^J \cap Q$, and $\widetilde{L}^J$, $\widetilde{Q}^J$ are their strict transforms.
 There is a  map from $(\ref{blownupcomplex})$ to $(\ref{coskeleton})$. The divisor $Q\cup L^J$ is normal crossing, so we can blow down once again to obtain an isomorphism of  $H^m(\widetilde{\Pro}^m \backslash \widetilde{Q}, \widetilde{L}^J \backslash \widetilde{Q}^J)$ with a motive we call
\begin{equation}\label{quotmotive}
h(\Delta_F) =  H^m( \Pro^m   \backslash Q, L^J \backslash Q^J) \ .  \end{equation}
There are  framings on $h(\Delta_F)$ given by $\omega_Q:\Q(-n)\overset{\sim}{\rightarrow}  \gr^W_{2n} H^m(\Pro^m\backslash Q) \cong \gr^W_{2n} h(\Delta_F)$, and
a   class $T_F : \Q(-r)\overset{\sim}{\rightarrow} \gr^W_{2r} h(\Delta_F)^{\vee}$, which we  define as follows.
Since $Q$ is smooth, the residue map onto $Q$ gives a morphism 
\begin{equation}\label{restoQ}
h(\Delta_F)\To H^{m-1}(Q,Q\cap L^J)(-1)\ ,
\end{equation}
and  gives an identification 
 $\Q(-r)\cong \gr^W_{2r} h(\Delta_F)\cong \gr^W_{2(r-1)} H^{m-1}(Q,Q\cap L^J)$, the latter being dual to  $H_{m-1}(Q,  Q\cap L^J)$. 
 An element in the Betti realization of this group is given by 
the  relative homology class of the spherical simplex $S_F \subset Q(\R)$ cut out by the hyperplanes $(L_j\cap Q)_{j\in J}$ on the set of real points of the quadric $Q$ (which can be identified as the boundary $\partial \Hyp^m$ in the Klein model), and $T_F$ is defined to be the corresponding framing on $ \gr^W_{2r} h(\Delta_F)$. It can be represented
by a tubular neighbourhood $N_F$ of $S_F$.
The sign is determined by the orientations, and 
one can verify that $T_F$ is dual to the class induced by $e_F$. 
Finally, the  residue of $\omega_Q$  along $Q$ is a certain rational multiple of the volume form on $Q$.

\begin{lem} Let $\Delta$ be  as above. For every strict odd-dimensional subface $F$ of $\Delta$,  which is not a one-dimensional face which contains the vertex at infinity, the coefficient $c_F$ of $\mot(F)$  in $(\ref{coprodoverF})$ is  equivalent to 
the framed mixed Tate motive $$[h(\Delta_F),[\omega_{\widetilde{Q}}], T_F](n)\ .$$
Its period is computed by integrating $\omega_{\widetilde{Q}}$ over the tubular neighbourhood $N_F$. By the residue formula and $(\ref{omegaq})$, this is   $2\pi i^n$ times a certain rational multiple of the the spherical volume of $S_F$. In particular, it lies in $i^n \R$. 
\end{lem}
By proposition $\ref{propmotiveofsimplex}$, the coproduct   $(\ref{coprodoverF})$  also contains contributions $(\ref{infmotivegraded2piece})$ from the  two-dimensional faces $G$ of $\Delta$ which contain the point at infinity.  Such a face $G$ defines a framed submotive $\mot(G)$ as before. To obtain the coefficients, consider two  distinct such faces $G_i$ and $G_j$. They intersect  along a line $G_i\cap G_j$
which contains the point at infinity, but nonetheless the complex  $(\ref{coskeleton})$ still defines a framed motive $h(\Delta_{G_i\cap G_j})$.
Inspection of the  spectral sequence which computes the hypercohomology of  $(\ref{coskeleton})$  shows that the framing $T_{G_i\cap G_j}$ corresponds via 
$(\ref{restoQ})$ to the dual of  $\pm(\alpha_{\widetilde{G}_i} -\alpha_{\widetilde{G}_j})$.   Thus the set of classes $T_{G_i\cap G_j}$  satisfy a single relation 
and define  a dual basis to  $\alpha(V_x\otimes_{\Q}\Q(-1))$ by  $(\ref{Vxasker})$. The period, as before, is proportional to  the dihedral angle subtended by $G_i$ and $G_j$.
The relation between the $T_{G_i\cap G_j}$  corresponds to the analytic fact that there is a single relation between the angles of the even-dimensional Euclidean simplex at infinity $\Delta_\infty$.

In all cases, the periods of the coefficient motives $c_F$ are  $\pi i^n$ times a rational multiple of the spherical volumes of the lunes $S_F$ (see also \cite{Go1} \S4.7).

\subsubsection{Volume and the main theorem}
The volume of a hyperbolic simplex with at most one vertex  at infinity is determined by its real period. The proof is identical to Goncharov's proof of the same
result  for finite simplices  \cite{Go1}.

\begin{cor}
Let $\Delta$ denote an oriented hyperbolic geodesic simplex in $\overline{\Hyp}^{2n-1}_{\overline{\Q}}$ with at most one vertex at infinity.
Then there is a number field $k$  and $\sigma:k\rightarrow \R$ such that $\mot(\Delta)\in \MT(k)$ and 
 $$\vol(\Delta) = (2\pi)^n R_\sigma (\mot(\Delta))\ .$$
\end{cor}
\begin{proof}
Let $P_\sigma$ be a period matrix for $\mot(\Delta)$. Its
 first column is given by integrating the form $\omega_{\widetilde{Q}}$ over a basis for the homology of $h(\Delta)$. The previous discussion
proves that, in the basis given by the classes $T_F$,  all entries lie in $i^{n} \R$ except for the last, which is given by $ i^{n-1} \vol(\Delta)$.
Thus, up to signs, $\Image ((2i\pi)^{-n}P_\sigma v_0)$
 is the column vector $(0,\ldots, 0, (2\pi)^{-n} \vol(\Delta))$, and
it follows from the definition and the choice of signs  on the framings that $R_\sigma[h(\Delta), \omega_{\widetilde{Q}}, \Delta]=\langle \Image ((2i\pi)^{-n} \comp_{\sigma,DR} \omega_{\widetilde{Q}}), \Real( [\Delta]) \rangle$
which is exactly
$(2\pi)^{-n}\vol(\Delta)\times 1.$
\end{proof}

Putting the various elements together,  we obtain the following theorem.

\begin{thm} \label{thmmotiveMAIN}
Let $M$ be a product-hyperbolic manifold modelled on $\Hyp^{2n_1-1}\times \ldots \times \Hyp^{2n_N-1}$, and defined over the fields $(k_1,\ldots, k_N)$. Then the framed
motive of $M$ is a well-defined element:
$$\mot(M) \in \Ext^1_{\MT(k_1)}(\Q(0),\Q(n_1))\otimes_{\Q} \ldots \otimes_{\Q} \Ext^1_{\MT(k_N)}(\Q(0),\Q(n_N))\ ,$$
such that the volume of $M$ is given by  the Hodge regulator:
$$\vol(M) = (2\pi)^{n_1+\ldots +n_N} R_\Sigma (\mot(M))\ .$$
If $M$ is equivariant with respect to $\Sigma=\{\sigma_1,\ldots, \sigma_N\}$, and $n_i=n$ for all $i$,  then
$$\mot(M) \in \textstyle{\bigwedge^\Sigma} \Ext^1_{\MT(k)}(\Q(0),\Q(n))\ .$$
\end{thm}

\subsection{Examples} \label{Sectexamples3space}

\subsubsection{A hyperbolic line element} First consider the simplest  case of a hyperbolic line segment $L$ in $\Hyp^1 \cong \R$. It corresponds to  a pair of points $\{x_0,x_1\} \in \Pro^1$, and a  quadric $Q\subset \Pro^1$ which consists of a pair of points $\{q_0,q_1\} \in \Pro^1$.
Then
\begin{equation} \label{Kummerline}
h(L) = H^1(\Pro^1\backslash \{q_0,q_1\}, \{x_0,x_1\})\ ,\end{equation}
 is  a Kummer motive, {\it i.e.}, $\gr^W_\bullet h(L)\cong \Q(0)\oplus \Q(-1)$. The  period defined by integrating the form $\omega_Q$ over the interval $[x_0,x_1]$ is 
$$\int_{x_0}^{x_1} {1\over 2}\Big({dx\over x-q_0}- {dx\over x-q_1}\Big) =
{1\over 2}\log\Big({(x_1-q_0)(x_0-q_1) \over  (x_0-q_0)(x_1-q_1)  } \Big)={1\over 2}\log \big([x_1\, x_0\big| q_0\, q_1]\big)\ .$$
The real period of $h(L)$ is half the real part of the logarithm  $\log \big|[x_1\,x_0|q_0,q_1]\big|$, which,
up to a sign, is the hyperbolic length of the oriented line segment $\{x_0,x_1\}$.

\subsubsection{Hyperbolic triangles in the hyperbolic plane}
Now consider a finite triangle $T$ in $\Hyp^2$. It defines a set of 3 lines $L_0,L_1,L_2$ and a smooth quadric $Q$ in $\Pro^2$. Then
$$h(T)= H^2\big(\Pro^2\backslash Q, \bigcup_{0\leq i\leq 2} L_i \backslash (L_i\cap Q)\big)\ .$$
From $(\ref{finitemotivegradedpieces})$, we have $\gr^W_2 h(T) = \Q(-1)^3$, and $\gr^W_0 h(T)=\Q(0)$.
Each side of $T$ defines a motive $(\ref{Kummerline})$ whose period is  its hyperbolic length. It follows that $h(T)$ splits as a direct sum
of Kummer motives, one corresponding to each side. The periods of $h(T)$  are  $2\ell_0,2\ell_1, 2\ell_2$ where $\ell_i$ is  the hyperbolic length of each side. Note that the triple $\{\ell_0,\ell_1,\ell_2\}$ is a complete isometry invariant for $T$.
Since $\gr^W_4 h(T)=0$, there is no interesting framed motive to speak of   in this case.

Now consider what happens when one vertex $x=L_1\cap L_2$ of $T$ is at infinity. After blowing up the point $x$, we obtain a smooth quadric $\widetilde{Q}$ in $\widetilde{\Pro}^2$,
and a configuration of four lines $\widetilde{L}_{-1}, \widetilde{L}_0, \widetilde{L}_1, \widetilde{L}_2$ as shown below (left). Now,
$$h(T) = H^2(\widetilde{\Pro}^2\backslash \widetilde{Q}, \bigcup_{-1\leq i\leq 2} \widetilde{L}_i\backslash (  \widetilde{L}_i \cap \widetilde{Q}))\ ,$$
and $(\ref{infmotivegraded2piece})$ gives $\gr^W_2 h(T)\cong \Q(-1)^2$, $\gr^W_0 h(T)= \Q(0)$. It is therefore the direct sum of two Kummer motives.
 One of them  comes from the inclusion of the side $L_0$
of finite length, whose period is  its hyperbolic length. To compute the period of the other, consider the following diagram.
\begin{figure}[h!]
  \begin{center}
    \epsfxsize=12.5cm \epsfbox{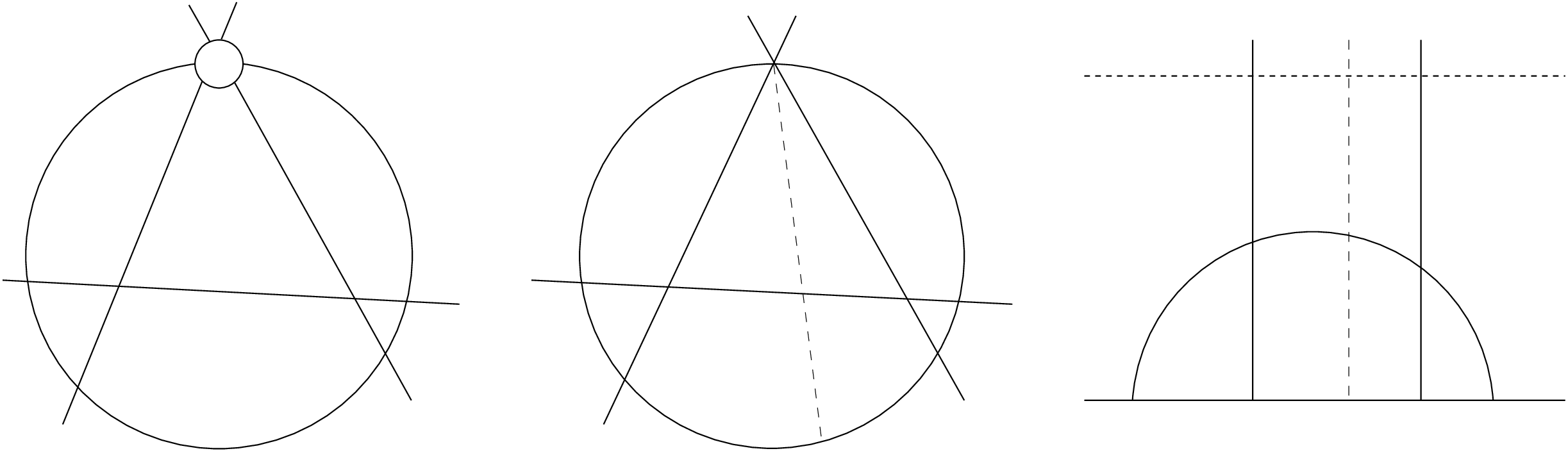}
\put(-278,85){${\tiny{\widetilde{Q}}}$}
\put(-280,55){${\tiny{\widetilde{L}_2}}$}
\put(-336,54){${\tiny{\widetilde{L}_1}}$}
\put(-308,22){${\tiny{\widetilde{L}_0}}$}
\put(-330,90){${\tiny{\widetilde{L}_{-1}}}$}
\put(-172,90){${\tiny{x}}$}
\put(-233,43){${\tiny{p}}$}
\put(-224,15){${\tiny{q}}$}
\put(-137,19){${\tiny{r}}$}
\put(-134,37){${\tiny{s}}$}
\put(-102,0){${\tiny{0}}$}
\put(-75,0){${\tiny{t_1}}$}
\put(-36,0){${\tiny{t_2}}$}
\put(-18,0){${\tiny{1}}$}
\put(-83,60){${\tiny{\widetilde{\ell}_1}}$}
\put(-28,60){${\tiny{\widetilde{\ell}_2}}$}
\put(-15,88){${\tiny{R}}$}
  \end{center}
\end{figure}
Note that the motive $h(T)$ is uniquely determined by the five points $x,p,q,r,s\in\partial \Hyp^2\cong \Pro^1(\R)$.

\begin{lem} The  periods of $h(T)$ are given by the two quantities
\begin{equation} \label{lxdef}
\ell_x={1\over 2}\log \Big({(x-q)^2 (p-r)(r-s) \over (x-r)^2(p-q)(q-s)}\Big) \quad \hbox{ and } \quad  \ell_0 = {1\over 2}\log \Big({(p-r)(q-s) \over (p-q)(r-s)}\Big)\ .
\end{equation}
\end{lem}
\begin{proof}
The periods of $h(T)$
define a Kummer variation 
on the configuration space of $5$ distinct points in $\Pro^1(\R)$ modulo the action of $\PSL_2(\R)$.
This is  the moduli space $\Mod_{0,5}(\R)$ of genus $0$ curves with 5 marked points. By projective transformation,
set $p=0,q=t_1,r=t_2,s=1,x=\infty$.
  The space of logarithms on $\Mod_{0,5}$ is spanned by $\log(t_1)$, $\log(t_2)$, $\log (1-t_1)$, $\log (1-t_2)$, $\log(t_2-t_1)$.
 The periods of $h(T)$ are additive with respect to subdivision (the dotted line above).
 A simple calculation shows that  the vector space of additive functions is spanned
by the two functions:
$$2\,\ell_x=\log\Big({t_2 (1-t_2) \over t_1(1-t_1)}\Big) \quad \hbox{ and } \quad 2\,\ell_0=\log\Big( {t_2(1-t_1)\over t_1(1-t_2)}\Big)\ .$$
Rewriting $t_1,t_2,1-t_1,1-t_2$  as cross-ratios, we obtain  formula $(\ref{lxdef})$.
\end{proof}
The quantity $\ell_x$ (resp. $\ell_0$)  is anti-invariant (resp. invariant) under the transformation $(p,q)\leftrightarrow (s,r)$.
One checks that $\ell_0$ is the hyperbolic length of the face defined by  $L_0$, and  $\ell_x$ is  the difference of  the regularised lengths
$\widetilde{\ell}_1-\widetilde{\ell}_2$ of the sides $L_1,L_2$. Here $\widetilde{\ell}_i$ is defined to be the length of the truncated line segment of side $L_i$
up to a horoball neighbourhood of $x$ (depicted by a horizontal dotted line at height $R$  in the figure above (right)).
The quantity $\widetilde{\ell}_1-\widetilde{\ell}_2$ is independent of $R$.

\subsubsection{A finite simplex in hyperbolic 3 space}
Consider a finite hyperbolic geodesic simplex $\Delta $ in $\Hyp^3$, which is given by four hyperplanes $L_0,\ldots, L_3$ in general position relative
to a smooth quadric $Q$ in $\Pro^3$. From $(\ref{finitemotivegradedpieces})$, we have
$$\gr^W_4 h(\Delta) \cong \Q(-2)\ , \ \quad \gr^W_2 h(\Delta) \cong \Q(-1)^6\ , \ \quad \gr^W_0 h(\Delta) \cong \Q(0)\ .$$
To compute the periods, choose an edge $L_{ij}$. It defines a complex of sheaves
\begin{equation}
\label{skeleton3}
\Q_{L_{ij}\backslash Q_{ij}}\rightarrow \bigoplus_{k\in \{0,1,2,3\}\backslash \{i,j\}} \Q_{L_{ijk}}\ ,
\end{equation}
whose hypercohomology is  the Kummer submotive of  the line element $L_{ij}$, relative to two points, whose real period is its hyperbolic length.
Now consider the analogue of    $(\ref{coskeleton})$ obtained from the set of faces containing the edge $L_{ij}$:
\begin{equation}
\label{coskeleton3}
\Q_{\Pro^{3}\backslash Q } \rightarrow  \Q_{L_{i}\backslash Q_{i}}  \oplus \Q_{ L_{j}\backslash Q_{j}}
\rightarrow   \Q_{L_{ij}\backslash Q_{ij}}\ .
\end{equation}
It defines  a Kummer motive  $H^3(\Pro^3\backslash Q, (L_i\cup L_j) \backslash Q)$  which maps to
$$H^3(\Pro^3\backslash Q, (L_i\cup L_j) \backslash Q) \To H^2(Q, (L_i\cup L_j)\cap Q)(-1)\ ,$$
via the residue. The period of the latter is computed as follows. The  two hyperplanes $L_i$, $L_j$  cut out a spherical lune on the quadric $Q(\R)$ (see figure 2).

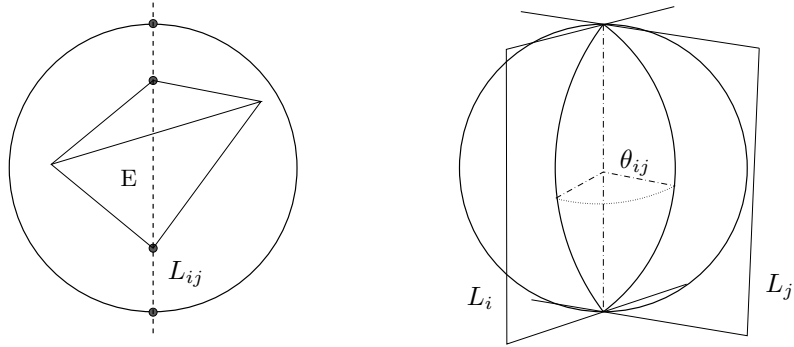
\begin{figure}[h!]
  \begin{center}
    \epsfxsize=10.0cm \epsfbox{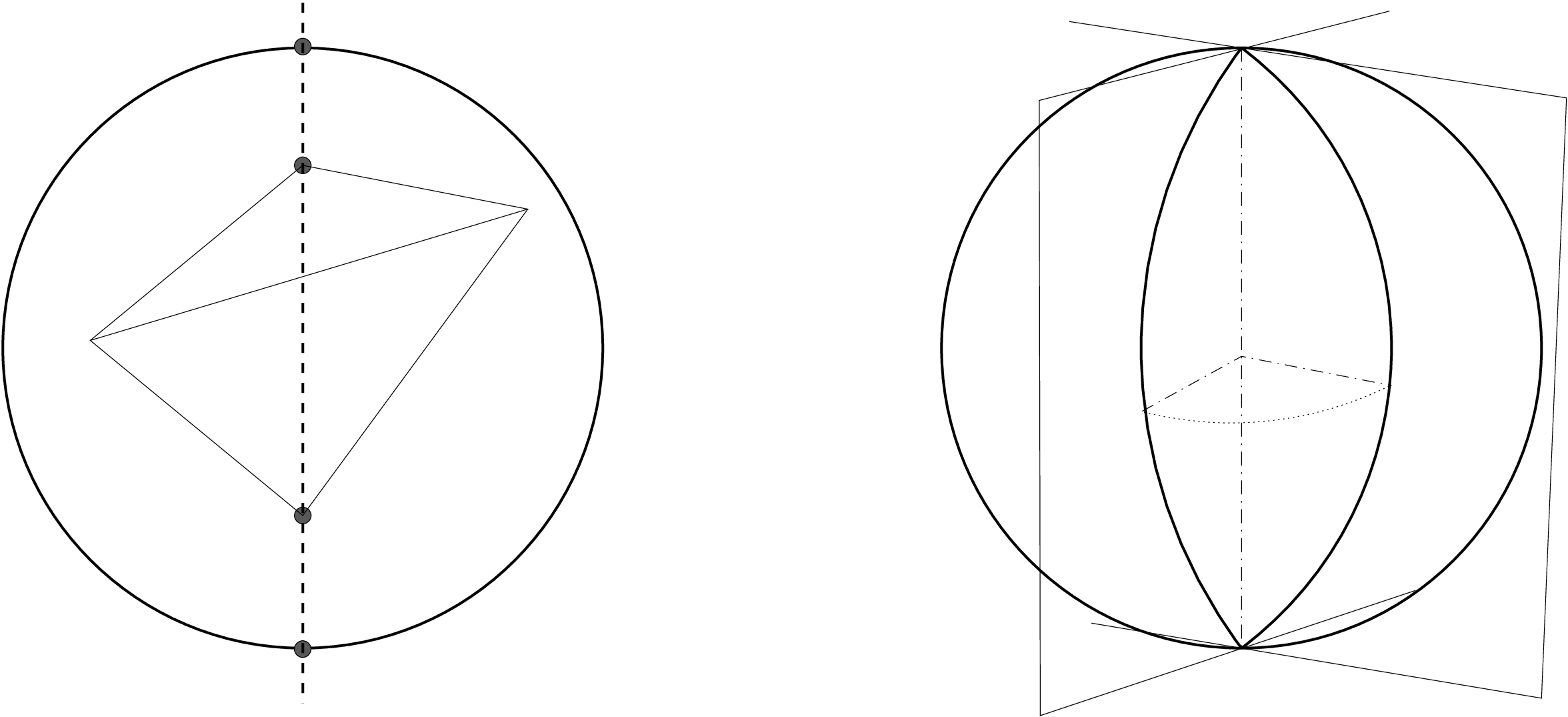}
\put(-242,60){\small{E}}
\put(-224,25){$\small{L_{ij}}$}
\put(-53,67){$\small{\theta_{ij}}$}
\put(-111,15){$\small{L_{i}}$}
\put(2,20){$\small{L_{j}}$}
 \label{Moving}
  \caption{Left: the complex $(\ref{skeleton3})$ corresponds to the edge $E$. Right: $L_i$ and $L_j$ define a spherical lune on $Q(\R)$ whose  real period is 
   twice the  dihedral angle  $\theta_{ij}$. }
  \end{center}
\end{figure}
The period is easily computed in spherical coordinates. Suppose that $Q$ is given by the affine equation $x^2+y^2+z^2=1$.  Then $\omega_Q= i(x^2+y^2+z^2-1)^{-2} dx\,dy\,dz$
is just $i\rho^2(1-\rho^2)^{-2} d\rho\, \sin (\phi) d\phi\, d\theta$. Its residue at $\rho=1$ is $i4^{-1}\sin (\phi)d\phi\, d\theta$,  which is $1/4i$ times the volume form on the sphere.
Thus the period obtained by integrating over the spherical lune is  ${1\over 2i} \theta_{ij}$, half  the dihedral angle between the hyperplanes $L_i,L_j$. 
If $X_{ij}$ is
a tubular
neighbourhood around $Q$ of the lune whose boundary is contained in $L_i \cup L_j$,  the relative homology classes $[X_{ij}]$ 
form a basis
for $\gr^W_2 h(\Delta)^\vee$.
In conclusion, we can write a (dual) period matrix for $h(\Delta)$ as follows:
$$\left(
    \begin{array}{ccccc}
      1 & 0 & \cdots & 0 & 0 \\
2\ell_{01} & 2i\pi & \cdots & 0 & 0 \\
\vdots &  & \ddots & & \vdots \\
2\ell_{23} & 0 & \cdots & 2i\pi & 0 \\
i\vol(\Delta) & 2\pi\theta_{01} & \cdots & 2\pi\theta_{23} & (2i\pi)^2 \\
\end{array}
  \right)
$$
where $\ell_{ij}$ is the hyperbolic length of the edge $L_{ij}$, for $0\leq i<j\leq 3$, and $\theta_{ij}$ is the dihedral angle subtended at that edge. The motive $\mot(\Delta)$ is obtained
from $h(\Delta)$ by a Tate twist by $\Q(3)$.
The reduced coproduct map $(\ref{coproddef1})$ on the level of period matrices, applied to $\mot(\Delta)$,  can therefore be written (ignoring Tate twists) as a Dehn invariant:
$$\sum_{0\leq i<j\leq 3} \left(
                                     \begin{array}{cc}
                                       1 & 0 \\
                                       2\ell_{ij} &  2i\pi  \\
                                     \end{array}
                                   \right)
\otimes
\left(
                                     \begin{array}{cc}
                                       1 & 0 \\
                                       i\,\theta_{ij} & 2 i\pi  \\
                                     \end{array}
                                   \right) \ .
 $$

\subsubsection{The case of a simplex in hyperbolic 3 space with a vertex at infinity}
Now consider the case when $\Delta$ has a single vertex at infinity $x=L_{123}$. After blowing-up this point, we obtain
 a new hyperplane $\widetilde{L}_{-1}$ which is the exceptional locus, and set
$$h(\Delta) = H^3(\widetilde{\Pro}^3\backslash {\widetilde{Q}}, \bigcup_{-1\leq i\leq 3} \widetilde{L}_i \backslash (\widetilde{Q}\cap \widetilde{L}_i) )\ .$$
From $(\ref{infmotivegraded2piece})$,  $\gr^W_4 h(\Delta)=\Q(-2)$, $\gr^W_2 h(\Delta)=\Q(-1)^5$ and $\gr^W_0 h(\Delta)=\Q(0)$.
The graded weight $2$ part is spanned by the Kummer submotives coming from each of the three finite-length edges $L_{01}, L_{02}, L_{03}$, whose periods are
their hyperbolic lengths, and a further 3 classes $e_{\widetilde{L}_{ij}}=\alpha_{L_{i}}-\alpha_{L_{j}}$, for  $1\leq i<j\leq 3$, where
$$\alpha_{L_i} \in H^2(\widetilde{L}_i\backslash \widetilde{Q}_i, \bigcup_{j\neq i} \widetilde{L}_{ij} \backslash \widetilde{Q}_{ij})\ ,  $$
is the class whose period is the quantity  defined in $(\ref{lxdef})$, on the left. The remaining periods are the dihedral angles $\theta_{ij}$ subtended at the edge $L_{ij}$ in all cases,
exactly as in the case where $\Delta$ is finite. There is a single relation between the classes $e_{\widetilde{L}_{ij}}$, and correspondingly, the angles
subtended at infinity  $\theta_{13}$, $\theta_{12}$, $\theta_{23}$ sum to $\pi$.

\section{Applications}
Let $M$ be a complete product-hyperbolic manifold of finite volume, which is modelled on a product of odd-dimensional hyperbolic spaces.
Then we can write $M=\X/\Gamma$, where $\X=\prod_{1\leq i\leq N} \Hyp^{2n_i-1}$, and $\Gamma$ is a discrete torsion-free subgroup of the group of automorphisms of $\X$.
By theorem $\ref{thmdefinedoverfield}$, $\Gamma$ is defined over  number fields $(k_1,\ldots, k_N)$. In $\S4$ and $\S5$, we constructed a framed motive
$$\mot(M) \in  \Am_{n_1}(k_1)\otimes_\Q \ldots \otimes_\Q \Am_{n_N}(k_N)\ .$$
Now let $\Gamma'$ denote another discrete torsion-free group acting on $\X$, which is commensurable with $\Gamma$, {\it i.e.}, $\Gamma\cap \Gamma'$ is of finite index
in both $\Gamma$ and $\Gamma'$. Then if $M'=\X/\Gamma'$,
\begin{equation}\label{commensmot}
\mot(M') =\mot(M)\, {[\Gamma:\Gamma\cap \Gamma'] \over [\Gamma':\Gamma\cap \Gamma']}\ .\end{equation}
This is clear from the construction: a tiling for
$\X/\Gamma\cap\Gamma'$ can be obtained by taking $[\Gamma:\Gamma\cap \Gamma']$, or
$[\Gamma':\Gamma\cap \Gamma']$, copies of a tiling for $M$, or $M'$ respectively.
In this way, one can define the motive of a product-hyperbolic orbifold. Let  $\Gamma$  denote  any discrete subgroup of automorphisms of $\X$ which is not necessarily torsion-free. After choosing
a torsion-free subgroup $\Gamma_0\leq \Gamma$ of finite index,  define
\begin{equation}\label{torsionmotive}
\mot(\X/\Gamma) = [\Gamma:\Gamma_0]^{-1}\, \mot(\X/\Gamma_0)\ .\end{equation}
If $M$ is isometric to a product  $M_1 \times M_2$, then $\mot(M) = \mot(M_1)\otimes \mot(M_2)$. 
We compute the motives for the cases when $M$ is an arithmetic manifold of type $(II)$, $(III)$ and omit the exceptional case $(IV)$.

\subsection{Dedekind Zeta motives for totally real number  fields.}
Let $k$ denote a totally real number field of degree $r$, and let $m=2n-1\geq 3$ be an odd integer.
Let $D$ be a quaternion algebra over $k$ satisfying the conditions of $(II)$, and let
 $Q$ be a skew-Hermitian form over $D$. Let  $d$ be the reduced norm of its discriminant and  set $L=k(\sqrt{d})$.
 Suppose that $Q$ has signature $(n,1)$ for $t$ places of $k$, where $t\geq 1$,  and is positive
definite for $r-t$ places. As in $\S3$, let
$$\Gamma \leq \U^+(Q,\Or)\quad \hbox{ and } \quad M_\Gamma=\Big( \prod_{i=1}^t \Hyp^n\Big)/\Gamma\ ,$$
where $\Gamma$ is any subgroup of finite index (not necessarily torsion-free).
 By remark $\ref{remtype1istype2}$, the corresponding cases $\Gamma\leq \SO^+(q,\Or)$ of type (I), $n$ odd, are subsumed in this construction.
 Let $\Sigma=\{\sigma_1,\ldots, \sigma_r\}$ denote the set of real embeddings of $k$.

There are two cases to consider.
If $L=k$, then the signature of ${}^{\sigma}Q$ must be the same for all embeddings $\sigma\in \Sigma$, and hence we must have $t=r$.
Otherwise,  $[L:k]=2$, and  $L$ has exactly $2t$ real embeddings, and $r-t$ pairs of complex conjugate embeddings.
Let $\tau \in \Gal(L/k)$ be a generator. Then the action of  $\tau$ gives an eigenspace decomposition:
\begin{equation}
\Ext^1_{\MT(L)}\big(\Q(0), \Q(n)\big)= E^+ \oplus E^-\ ,
\end{equation}
where $E^+ \cong \Ext^1_{\MT(k)}(\Q(0), \Q(n))$. Let $\chi$ denote the non-trivial character of $\Gal(L/k)$, and let
$L(\chi,s)=\zeta_L(s) \zeta_k(s)^{-1}$ denote the corresponding Artin $L$-function.

\begin{thm}\label{thmDedZetamotive} Let $\Gamma, M_\Gamma$ be as above. Let $\mot(M_\Gamma)$ denote the framed motive corresponding to $M_\Gamma$ as defined
by $(\ref{torsionmotive})$. If $k=L$ then
\begin{equation} \label{thmDZM1}
\mot(M_\Gamma) \in  \textstyle{\bigwedge^\Sigma} \,\Ext^1_{\MT(k)}(\Q(0),\Q(n))\ , \end{equation}
and there exists a non-zero rational number $\alpha$ such that $(2\pi)^{-mr}\vol(M_\Gamma)$  is
\begin{equation} \label{thmDZM2}
R_\Sigma(\mot(M_\Gamma)) =\alpha\, \zeta^*_k(1-m) \ .\end{equation}
Otherwise, in the case where $[L:k]=2$,
\begin{equation}\label{thmDZM3}
\mot(M_\Gamma) \in \textstyle{\bigwedge^\Sigma}\, E^- = \textstyle{\bigwedge^\Sigma} \,(\tau-1)\Ext^1_{\MT(L)}(\Q(0),\Q(n))\ , \end{equation}
and there exists a non-zero rational number $\alpha$ such that $(2\pi)^{-mt}\vol(M_\Gamma)$  is
\begin{equation}\label{thmDZM4}
R_\Sigma(\mot(M_\Gamma)) =\alpha\, L^*(\chi,1-m) \ .\end{equation}
\end{thm}

\begin{proof}
In the first case when $k=L$,  the manifold $M_\Gamma$ is defined over $(k_1,\ldots, k_r)$, where $k_i=\sigma_i k$ for $1\leq i\leq r$, and
is equivariant by definition. Then $(\ref{thmDZM1})$ follows from theorem $\ref{thmmotiveMAIN}$, and
$(\ref{thmDZM2})$ follows from corollary  $\ref{corvolumeslist}$.

 In the second case, when $[L:k]=2$, the manifold 
$M_\Gamma$ is defined over $(L_1,\ldots, L_r)$, where $L_i=\sigma_i k(\sqrt{d})$, and is also equivariant with respect to $\Sigma$. Let $\tau$ be a generator
of $\Gal(L/k)$. By construction, it maps a tiling of  $M_\Gamma$ to another tiling of $M_\Gamma$, and therefore
preserves the framing given by the relative homology classes of the simplices defining $\mot \, M$.  However $\tau$ clearly acts with sign $-1$ on the volume form $(\ref{omegaq})$, and therefore
the framing in $\gr^W_0(\mot(M_\Gamma))$ is anti-invariant under $\tau$. We deduce that
$\tau(\mot(M_\Gamma))= -\mot(M_\Gamma).$
Now $(\ref{thmDZM3})$ and $(\ref{thmDZM4})$
follow from theorem $\ref{thmmotiveMAIN}$ and corollary  $\ref{corvolumeslist}$ as before.
\end{proof}
Note that since $\vol(M_\Gamma)$  is non-vanishing,  theorem $\ref{thmDedZetamotive}$ implies, ``independently" of Borel's theorem $(\ref{Ktheorydimensions})$, that
$\dim_\Q E^+\geq r$ and $\dim_\Q E^-\geq t$.
\begin{defn} Many different discrete groups $\Gamma$ give rise to the same framed motive, up to multiplication by a rational number.
For number-theoretic applications, we can take simple representatives for $\Gamma$ as follows. Let $d\in k^{\times}$, and set
$$q(x_0,\ldots, x_{2n})=-dx_0^2+x_1^2+\ldots+x_{2n}^2\ .$$
We define $\mot(k,d,n) = \mot(M_\Gamma)$, where $\Gamma = \SO^+(q,\Or_k)$.
\end{defn}

\begin{cor} \label{cornegvals} As above,  let $k$ be a totally real number field, let $n>1$ be odd,  let $L=k(\sqrt{d})$, where $[L:k]=2$ and $d\in k$ is positive  for at least one embedding of $k$.
Let $\Sigma_k$ denote the set of real embeddings of $k$.
Then
$$\Lo_k=\textstyle{\bigwedge^{\!\Sigma_k}}\, \Ext^1_{\MT(k)}(\Q(0),\Q(n))\ , \quad  \Lo'_L=  \textstyle{\bigwedge^{\!\Sigma_k}}\,(1-\tau) \Ext^1_{\MT(L)}(\Q(0),\Q(n)) \ ,$$
are 1-dimensional $\Q$-vector spaces, and  $\mot(k,1,n)\in \Lo_k$ and $\mot(k, d, n)\in \Lo'_L $
are generators. 
Up to a  rational factor, the Hodge regulator on each element is, respectively,
$\zeta_k^*(1-n)$, $L^*(\chi,1-n)$, where $\chi$ is a non-trivial character of $Gal(L/k)$. 
\end{cor}
\begin{proof}
By $(\ref{Ktheorydimensions})$, we have $\dim_\Q E^+ = r $,  and $\dim_\Q E^- = t$.\end{proof} 
\begin{cor} \label{thmzetakisdethyp} The special value  $\pi^{nr}\zeta^*_k(1-n)$ is a determinant of sums of volumes of hyperbolic simplices defined over $k$ with at most one vertex at infinity.
\end{cor}

This corollary uses Borel's bound for the rank of algebraic $K$-groups  and does not follow  directly  from
 a   decomposition of $M_\Gamma$ into simplices.

\subsection{Quadric motives and generators for $\Ext^1_{\MT(k)}(\Q(0),\Q(n))$}
Let $k$ be a totally real number field, and let $L=k(\sqrt{d})$ where  $d\in k^\times$ is positive for at least one embedding of $k$. Let us fix a smooth quadric $Q_d$ in $\Pro^{2n-1}$  by:
$$Q_d=\{-d x_0^2+x_1^2+\ldots +x_{2n-1}^2=0\}\ .$$
\begin{defn}
Let $L_1,\ldots, L_{2n}$ denote a set of hyperplanes in general position and defined over  $k$.  Define a finite quadric motive over $k$ to be:
\begin{equation}\label{finitequadmotdef}
m(Q_d,L) = H^{2n-1}(\Pro^{2n-1}\backslash Q_d, \bigcup_{1\leq i\leq 2n} L_i\backslash (L_i \cap Q_d) )\in \MT(\overline{\Q})\ ,
\end{equation}
with its framing as defined in $\S5$. 
In the case where $Q_d,L_1,\ldots, L_{2n}$ do not cross normally, we blow up the points $z_i=L_1\cap\ldots\cap \widehat{L_i} \cap \ldots \cap L_{2n}$ which meet $Q_d$.
Let $\widetilde{\Pro}^{2n-1}$ denote the blow-up of $\Pro^{2n-1}$ in $\{z_i: 1\leq i\leq 2n | z_i \in Q_d\}$, and let $\widetilde{L}_i, \widetilde{Q}_d$ denote the strict
transforms of $L_i$, $Q_d$. Define a quadric motive in this case to be:
\begin{equation}\label{infinitequadmotdef}
m(Q_d,L) = H^{2n-1}(\widetilde{\Pro}^{2n-1}\backslash \widetilde{Q}_d, \bigcup_{1\leq i\leq 2n} \widetilde{L}_i\backslash (\widetilde{L}_i \cap \widetilde{Q}_d) )\in \MT(\overline{\Q})\ ,
\end{equation}
with its  framing as given in $\S5$.
\end{defn}

For each embedding $\sigma$ of $k$ into $\R$ for which
 $\sigma(d)$ is positive, the  points $\sigma(x_i)$  define
a  hyperbolic geodesic simplex in the Klein model which has finite volume.

\begin{thm} \label{thmquaddecomp} Let $d, k, n$ be as above.
 Every  $M \in \Ext^1_{\MT(L)} (\Q(0), \Q(n))$, viewed as an element of $\Am_n(L)$, is equivalent  to a linear combination of quadric motives:
\begin{equation}\label{Mquaddecomp}
M = \sum_{i=1}^R m(Q_d,L_i)\ ,
\end{equation}
and its real periods are $(2\pi)^{-n}$ times the  corresponding sum of hyperbolic volumes.
\end{thm}

\begin{proof}
Let $\Sigma$ be the set of places of $k$ and assume that $[L:k]=1$, i.e., $d\in k^{\times 2}$. Let $\overline{V}\subset \Am_n(\overline{\Q})$ denote the $\Q$-vector space spanned by all linear combinations of quadric motives $(\ref{infinitequadmotdef})$ in $\Pro^{2n-1}$. Let  $V= \overline{V}\cap \Am_n(k)$ and  let
$$V_0 \subset \ker \Big(\widetilde{\Delta}_n: \Am_n(k) \To \bigoplus_{1\leq r\leq n-1} \Am_r(k)\otimes_\Q \Am_{n-r}(k)\Big) $$
denote the subspace of $V$ on which the  reduced coproduct vanishes. The right-hand side of the previous line is isomorphic to  $E^+=\Ext^1_{\MT(k)}(\Q(0),\Q(n))$,
and we identify $V_0$ with its image in $E^+$.
We have  construced an   element
$$\mot(k,1,n) \in \textstyle{\bigwedge^{\!\Sigma}}\, V_0 \subset \textstyle{\bigwedge^{\!\Sigma}}\, E^+\ ,$$
which is non-zero since its regulator does not vanish, by corollary \ref{cornegvals}.
It follows that $\dim_\Q V_0\geq r$, and by Borel's theorem, 
$\dim_\Q E^+= r$. Therefore $V_0=E^+$, and every element in $\Ext^1_{\MT(k)}(\Q(0),\Q(n))$ is a linear combination of quadric motives over $k$.
Now if $L=k(\sqrt{d})$ and $[L:k]=2$, the same argument applied to $\mot(k,d,n)$ shows that every element
in $E^-$ is a linear combination of  motives $(\ref{infinitequadmotdef})$ with $d\notin k^{\times2}$. The theorem follows from  the fact that
$\Ext^1_{\MT(L)}(\Q(0),\Q(n))\cong  E^+\oplus E^-$.
\end{proof}

The proof of theorem $\ref{thmmotiveMAIN}$ only requires hyperbolic simplices with at most one vertex at infinity. 
 We can therefore
 assume that all quadric motives which occur have at most one vertex $x_i$ lying on $Q_d$.   This proves theorem $\ref{Intromaintheorem2}$.

 One can show using 
   a theorem due  to  Sah \cite{Sah1} that it suffices to consider finite simplices only, since  a hyperbolic geodesic simplex with vertices at infinity is stably scissors-congruent to a sum of finite ones.

\subsection{The case $n=2$ and Zagier's conjecture}\label{sectZagiersconj}
The case of hyperbolic $3$-space is different  because of the exceptional isomorphism $\SO^+(3,1)(\R)\cong \PSL_2(\C)$, and  as a result, the analogues of the previous
theorems hold  for all number fields, not just totally real ones. We can also use ideal triangulations in this case \cite{NY}.

Let us identify the boundary of hyperbolic 3-space with the complex projective line: $\partial \Hyp^3\cong \Pro^1(\C)$, with the  action
of $\PSL_2(\C)$ by M\"obius transformations. An ideal hyperbolic 3-simplex is given by 4 distinct
points on  $\partial\Hyp^3$, and by projective transformation, we can assume that 3 of them are at $0,1$ and $\infty$, and denote the last point by $z\in \Pro^1\backslash \{0,1,\infty\}$. If $z$ lies in a number field $L'\subset \C$,  then $\Delta(0,1,\infty, z)$ defines a framed mixed Tate motive  by
  $(\ref{infinitequadmotdef})$ with graded pieces $\Q(0),\Q(-1), \Q(-2)$. One can verify that it is defined over the field $L'$.

Now consider an arithmetic  group $\Gamma$ of type (III). So let $L$ be a number field of degree $r$ with $r_1$ real places and
$r_2$ complex places, where $r_2\geq 1$,  and let $0\leq t\leq r_1$. Let $B$ denote a quaternion algebra over $L$ which is unramified at $t$ real places, and ramified
at $r_1-t$ real places. For any order $\Or$ in $B$, let $\Gamma$ denote a subgroup of finite index of the elements of $\Or$ of reduced norm 1.
Let $\Sigma$ denote the set of complex places of $L$. By triangulating over a suitable splitting field using ideal hyperbolic 3-simplices (see \cite{NY}), we obtain a framed motive
$\mot(M_\Gamma)$  as before.
\begin{thm}\label{thmzeta2} The element
$\mot(M_\Gamma) \in \textstyle{\bigwedge^{\!\Sigma}}\,\Ext^1_{\MT(L)}(\Q(0),\Q(2))$ satisfies
$$R_{\Sigma} \big(\mot(M_\Gamma)\big)= \alpha \,\zeta_L^*(-1) \quad \hbox{ for some } \alpha \in \Q^{\times}\ .$$
\end{thm}
Borel's theorem  $(\ref{Ktheorydimensions}$) implies that the rank of $\textstyle{\bigwedge^{\!\Sigma}}\,\Ext^1_{\MT(L)}(\Q(0),\Q(2))$ is exactly $r_2$,
and hence it is generated by the class $\mot(M_\Gamma)$.
More simply, one can take $\Gamma$ to be precisely $\PSL_2(\Or_L)$, where $\Or_L$ is the ring of integers of $L$. Then
$\Gamma$ acts on $(\Hyp^3)^{r_2}$ and the quotient is a Bianchi orbifold. We denote its motive by $\mot(L,2)$.
\begin{cor} The element $\mot(L,2)$ is a generator of $\textstyle{\bigwedge^{\!\Sigma}}\,\Ext^1_{\MT(L)}(\Q(0),\Q(2))$.
\end{cor}
The analogue of theorem $\ref{thmquaddecomp}$ is the following.

\begin{cor} Every element $M\in\Ext^1_{\MT(L)}(\Q(0),\Q(2))$ is framed equivalent to a linear combination of motives of hyperbolic simplices
$\Delta(0,1,\infty,z)$ where $z \in L$.
\end{cor}
A closer analysis of the gluing equations between hyperbolic 3-simplices (or, by looking at the coproduct on the corresponding framed motives, which can be written
explicitly), one verifies that every such element is a linear combination $\sum_{i=1}^R n_i\, \mot(\Delta(0,1,\infty,{z_i}))$ where $n_i \in \Z$, and the elements $z_i \in L\backslash \{0,1\}$ satisfy:
\begin{equation}\label{admissible}
\sum_{i=1}^R n_i z_i\wedge (1-z_i) =0 \in \textstyle{\bigwedge^2} L^\times\ .\end{equation}
Using the well-known fact (see below) that the volume of $\Delta(0,1,\infty,z)$ is given by the Bloch-Wigner dilogarithm
$D(z) =\Image( \mathrm{Li}_2(z) + \log |z|\log(1-z)),$
 the analogue of theorem $\ref{thmzetakisdethyp}$ is precisely Zagier's conjecture for $n=2$.

\begin{cor}
There exist finite formal linear combinations  $\xi_j=\sum n_{i,j} [z_{i,j}]$ for $1\leq i\leq r$, where $n_{i,j} \in \Z, z_{i,j} \in L$ satisfying $(\ref{admissible})$, and a non-zero rational number $\alpha$
such that:
$$\zeta^*_L(-1) = \alpha \det(D(\sigma_i(\xi_j)))\ .$$
\end{cor}

\begin{rem}
This corollary is well-known by the work of Bloch, Suslin and many others, by  relating  the Borel regulator on $K_3(L)\otimes \Q$ to the Bloch-Wigner dilogarithm $D$ on the Bloch group of $L$. The relation with the zeta-value comes from Borel's theorem $(\ref{introBorelzeta})$, which is a separate argument.  Note that in the proof outlined above, Zagier's conjecture and theorem $\ref{thmzeta2}$ are proved by the same argument, 
since quadric motives in dimension 3 are equivalent to dilogarithmic motives $P_2(z)$.
\end{rem}

\subsubsection{The volume of an ideal hyperbolic 3-simplex}
The moduli space of ideal 3-simplices is
$\Mod_{0,4} \cong \Pro^1\backslash \{0,1,\infty\},$ parameterized by a single coordinate $z$. The volume is a function of $z$ we denote by $v(z)$.
By definition $(\ref{infinitequadmotdef})$, each ideal 3-simplex defines a mixed Tate motive with graded pieces $\Q(0)$, $\Q(-1)$ and $\Q(-2)$ only. It
therefore defines a
unipotent variation of mixed Hodge structure on $\Mod_{0,4}$. In particular, the function $v(z)$ satisfies the following properties:
\begin{enumerate}
  \item It is a unipotent function on $\Mod_{0,4}$ of weight 2.
  \item It is single-valued and extends continuously to $\overline{\Mod}_{0,4}\cong \Pro^1$.
  \item It vanishes at the points $0,1$ and $\infty$.
\end{enumerate}
The algebra of all single-valued unipotent functions on $\Mod_{0,4}$ was explicitly constructed in \cite{Br}. A  basis for the vector-space of such  functions is $\Lo_w(z)$, where
$w$ is a  word in the alphabet on two letters $x_0,x_1$.
 It follows that an arbitrary  single-valued unipotent function $F$ of weight 2 is of the form:
$$F(z)= a_{x_0^2}\,\Lo_{x_0^2}(z) + a_{x_0x_1}\,\Lo_{x_0x_1}(z) +a_{x_1x_0}\,\Lo_{x_1x_0}(z) +a_{x_1^2}\,\Lo_{x_1^2}(z) \qquad a_w\in \R\ ,$$
where $\Lo_{x_0^n}(z) = {1 \over n} \log^n |z|^2$, $\Lo_{x_1^n}(z) = {1\over n} \log^n|1-z|^2$, and, for example,
$$\Lo_{x_0x_1}(z) = 2i \,\Image( \mathrm{Li}_2(z) + \log |z|\log(1-z))-2\log|z|\log|1-z|\ .$$
One also has the shuffle relation $\Lo_{x_0}(z)\Lo_{x_1}(z) = \Lo_{x_0x_1}(z) + \Lo_{x_1x_0}(z)$, and many other properties \cite{Br}. From property (3), we conclude
that $v(z)= a (\Lo_{x_0x_1}(z)- \Lo_{x_1x_0}(z))$. The constant $a=(4i)^{-1}$ is easily determined from a special case, giving
$$v(z) =\Image( \mathrm{Li}_2(z) + \log |z|\log(1-z))\ ,$$
which is none other than the Bloch-Wigner dilogarithm. The same result has been proved many times over by a wide variety of methods.

\subsubsection{Volumes of  higher-dimensional  hyperbolic simplices} \label{sectvolhigher} A similar variational argument allows one to show that the volumes of 
hyperbolic simplices in $\Hyp^{2n-1}$ are expressible in terms of (logarithms and products of) multiple polylogarithms
$$\mathrm{Li}_{n_1,\ldots, n_r}(x_1,\ldots, x_r) = \sum_{1\leq k_1 < \ldots < k_r} {x_1^{k_1}\ldots x_r^{k_r} \over k_1^{n_1} \ldots k_r^{n_r}}\ ,$$
of  weight   $n=n_1+\ldots+n_r.$ 
It is possible to construct single-valued versions of these functions in the same way as for the multiple polylogarithms in one variable  \cite{Br}.
By  corollary $\ref{thmzetakisdethyp}$, we deduce that the special  value $\zeta^*_k(1-n)$, where $k$ is a totally real number field, is expressible in terms of powers of $\pi$ and a determinant of values of multiple polylogarithms evaluated in possibly an extension of $k$.

In the case $n=3$, it is relatively easy to show that every multiple polylogarithm  of weight 3 can be expressed in terms of the trilogarithm $\mathrm{Li}_3(z)$. From this fact, one could presumably  obtain a proof of Zagier's conjecture for $n=3$ for totally real number fields.  The case $n=4$ is the first  interesting case, since 
every multiple polylogarithm of weight 4 can be expressed in terms of two functions:
$$ \mathrm{Li}_4(z) \qquad \hbox{ and } \qquad \mathrm{Li}_{2,2}(x,y)$$
 since $\mathrm{Li}_4(z)$ does not suffice on its own. The expectation is that if $M$ is an arithmetic product-hyperbolic manifold modelled on products of $\Hyp^{7}$, then the  vanishing of the reduced coproduct for $\mot(M)$ (i.e.,  the gluing equations for the simplices) should imply that  the sum of the  volume terms of the form $\mathrm{Li}_{2,2}(x,y)$ over all simplices in  the triangulation should be expressible in terms of $\mathrm{Li}_4(z)$.  In short, $M$ should define a functional equation for combinations of $\mathrm{Li}_{2,2}(x,y)$ in terms of $\mathrm{Li}_4(z)$.

\subsection{Some open questions} \label{sectOpen}
\begin{enumerate}
  \item  Since our proof of theorem $\ref{Intromaintheorem}$ is motivic, it is natural to ask what happens in other realisations, and in particular the $p$-adic case.
  \item We have only considered arithmetic product-hyperbolic manifolds, but there exists an abundance of non-arithmetic manifolds $\Hyp^n/\Gamma$, which define elements in $\Ext^1_{\MT(\overline{\Q})}(\Q(0),\Q(n))$.
  Does there exist a volume  formula for non-arithmetic hyperbolic manifolds $M=\Hyp^n/\Gamma$ as a Dirichlet-type series?  To the author's knowledge, there does not seem to exist a formula, conjectured or otherwise,  for  the regulator on single elements of $\Ext^1_{\MT(k)}(\Q(0),\Q(n))$
 in the case where $k$ is a non-abelian Galois extension of $\Q$.
  \item Does the construction of $\mot_n(k)$  generalise to   other symmetric spaces? The case of the special linear group would 
  give a result for  values of Dedekind zeta functions for all number fields,  not just the totally real ones.
  \item Can one  construct non-trivial iterated extensions in the category $MT(k)$ by considering manifolds with boundary? One can dream of proving the freeness conjecture (\S\ref{sectintroexpl})  by defining nested families of manifolds with boundary yielding elements of  $\Am_n(k)$ which are coindependent by construction. The simplest non-trivial example  would be a manifold whose period is a multiple of  $\zeta(3,5)$, and its boundary would relate in some way  to the corresponding manifolds for  $\zeta(3)$ and $\zeta(5)$.
  \end{enumerate}

\section{Example: A Coxeter motive for $L(\chi,3)$}
The following  example of a fundamental domain, for an arithmetic reflection group acting on $\Hyp^5$, is due to Bugaenko \cite{Bug}.
Only exceptionally few examples can hope to  have such  a simple and explicit description.
 Let $\phi={1+\sqrt{5} \over 2}$, and let $k=\Q(\sqrt{5})$. Its ring of integers $\Or_k$ is $\Z[\phi]$.  Consider the quadratic form:
$$q(x_0,\ldots, x_5)=-\phi\, x_0^2 +x_1^2+\ldots +x_5^2\ ,$$
and let $\Gamma=\SO^+(\Or_k, q)$ be the group of  $\Or_k$-valued matrices which preserve  $q$ and which map each
 connected component of $\{x: \,q(x)<0\}$ to itself. Then $\Gamma$ is a discrete group of automorphisms of $\Hyp^5$ of type (I) as defined in $\S2$.
Note, however, that $\Gamma$ has torsion. Consider the seven hyperplanes in $\Pro^5$:
$$\begin{array}{crll}\label{6hyps}
L_1\ :& x_2-x_1&=& 0\ ,\nonumber \\
L_2\ :&x_3-x_2&=&0\ , \nonumber \\
L_3\ :&x_4-x_3&=& 0\ ,\nonumber \\
L_4\ :&x_5-x_4&=& 0\ ,\nonumber \\
L_5\ :&x_5&=&0 \ ,\nonumber \\
L_6\ :& (\phi-1)\,x_0+\phi \,x_1&= &0 \ ,\\
L_7\ :&(1+\phi)\, x_0 + \phi\,(x_1+x_2+x_3+x_4+x_5)&=&0 \nonumber\ .
\end{array}
$$
These hyperplanes bound a convex polytope $P$.  If $Q=\{x\in \Pro^5: q(x)=0\}$ denotes the quadric defined by $q$, the set of real points of $\Pro^5\backslash Q$ (more precisely
$\{x\in \Pro^5(\R): q(x)<0\}$) is a projective model for $\Hyp^5$.
One proves \cite{Bug}, that the group generated by hyperbolic reflections in the $L_i$, for $1\leq i\leq 7$
generates $\Gamma$, and therefore that the interior of $P$ is an open fundamental domain for $\Gamma$. The Coxeter diagram for $\Gamma$ is the following:

\begin{picture}(70,40)(-100,-20)
\multiput(2,0)(20,0){4}{\line(1,0){16}}

\multiput(0,0)(20,0){7}{\circle{4}}

\put(1,2){\line(1,0){18}}
\put(1,-2){\line(1,0){18}}

\put(81,2){\line(1,0){18}}
\put(81,-2){\line(1,0){18}}

\put(102,0){\line(1,0){4}}
\put(108,0){\line(1,0){4}}
\put(114,0){\line(1,0){4}}

\put(-2,6){\tiny{6}}
\put(18,6){\tiny{1}}
\put(38,6){\tiny{2}}
\put(58,6){\tiny{3}}
\put(78,6){\tiny{4}}
\put(98,6){\tiny{5}}
\put(118,6){\tiny{7}}

\end{picture}

\noindent
The polytope $P$ has the combinatorial structure of a prism, {\it i.e.}, the product of a 5-simplex with an interval, and has no non-trivial symmetries.

Now consider the motive
$$h(\Gamma)=H^3\big(\Pro^5\backslash Q, \textstyle{ \bigcup_{i=1}^7} L_i \backslash (L_i\cap Q)\big)\in \MT(\overline{\Q})\ . $$
It has canonical framings  $[P]\in \gr^W_6 H_3(\Pro^5, \bigcup_{i=1}^7 L_i) \cong (\gr^W_0 h(\Gamma))^\vee $ given by the class of the  polytope $P$, and by the class of the
 volume form
$$\omega_Q= \sqrt{\phi} \,{\sum_{i=0}^5 (-1)^i x_i dx_0\wedge \ldots  \widehat{dx_i} \ldots  \wedge dx_5 \over q(x_0,\ldots, x_5)^3} \ ,\quad  \hbox{ and } \quad [\omega_Q]
 \in \gr^W_6 H^3(\Pro^5\backslash Q) \cong \gr^W_6 h(\Gamma) .$$
Let $\Gamma'$ denote a torsion-free subgroup of $\Gamma$ of finite index. We can construct a fundamental domain for $\Gamma'$ by gluing $[\Gamma:\Gamma']$ polytopes $P$ together. By $\S4$, we know that the framed equivalence class of the corresponding motive, which is
an integral multiple of $[h(\Gamma), [P], [\omega_Q] ]$, has vanishing coproduct. It therefore defines an extension of $\Q(-3)$ by $\Q(0)$, and hence its Tate twist:
 $$\mot(\Gamma)=[h(\Gamma), [P], [\omega_Q] ] (3)\in \Ext^1_{\MT(\overline{\Q})}(\Q(0),\Q(3))\ .$$
The motive $h(\Gamma)$ is clearly defined over $k$, and its framing is defined over $L=k(\sqrt{\phi})$.
It is clear, therefore, that $\mot(\Gamma) \in  \Ext^1_{\MT(L)}(\Q(0),\Q(3))$,
and furthermore, that it is anti-invariant under the action of $\tau$, the non-trivial generator of  $\Gal(L/k)$.
This is because $\tau$ fixes $Q$, the $L_i$ and $[P]$, but sends $[\omega_Q]$ to $-[\omega_Q]$. Thus
\begin{equation}
\mot(\Gamma) \in (1-\tau)\, \Ext^1_{\MT(L)}(\Q(0),\Q(3))\ .
\end{equation}
By $(\ref{extgroups})$  and $(\ref{Ktheorydimensions})$, the right-hand side is a $\Q$-vector space of dimension $1$, and the framed motive $\mot(\Gamma)$ is therefore a generator.

Next, by the volume computations given in $\S3$, we deduce that
\begin{equation} \label{appendixvol}
\vol(\Hyp^5/\Gamma') \sim_{\Q^\times}   \sqrt{|d_{L/k} d_k|} {L(\chi,3) \over \pi^3} \sim_{\Q^\times}\pi^3 L^*(\chi,-2)\ , \end{equation}
where $\chi$ is the non-trivial quadratic character of $\Gal(L/k)$.
We conclude using the fact that $\vol(\Hyp^5/\Gamma) = \int_{[P]} [\omega]$ is the real period of $\mot(\Gamma)$.
\begin{cor} We have explicitly constructed  $\mot(\Gamma) \in \Ext^1_{\MT(L)}\big(\Q(0),\Q(3)\big) $ such that
$\tau(\mot(\Gamma))=-\mot(\Gamma)$ and
$L^*(\chi,-2) =  R_\tau\,  \mot(\Gamma) ,$ where $R_\tau$ is the real period map (Hodge regulator).
\end{cor}
Note that it is possible in principle to compute the exact volume of $P$, using the methods of \cite{Pr}.
It is  also known that the volume of a hyperbolic 5-simplex can be written as a linear combination of values of (single-valued) trilogarithms (\S\ref{sectvolhigher}).
\begin{rem} For the main result of the paper we need instead to take the group of $\Or_k$-automorphisms of the quadratic form $q(x)=-x_0^2+x_1^2+\ldots +x_5^2$
whose automorphism group  now acts on the pair of hyperbolic spaces $\Hyp^5\times \Hyp^5$.
Finding an explicit fundamental domain in this case should be possible
using the methods of Epstein and Penner \cite{EP}, but is complicated in practice.

The above example  has a very simple fundamental domain precisely  because it is a reflection group. Reflection groups are known only to exist in hyperbolic
spaces $\Hyp^n$ for bounded $n$ and for number fields of bounded discriminant (when $n\geq 4$). 
\end{rem}

\end{document}